\documentclass{amsart}

\usepackage[latin1]{inputenc}
\usepackage{amsmath, amsthm, amssymb, booktabs, multirow, graphicx,
  booktabs, bm, float}
\usepackage{mathtools}
\usepackage{enumitem}
\usepackage{dialogue}
\usepackage[stable]{footmisc}
\usepackage{tikz,nicefrac}
\usepackage{subcaption}
\usetikzlibrary{calc,matrix,arrows,decorations.markings}
\usepackage[colorlinks]{hyperref}
\usetikzlibrary{
  graphs,
  graphs.standard
}
\usepackage{url}
\usepackage{todonotes}

\newtheorem{defin}{Definition}[section]

\newtheorem{theorem}[defin]{Theorem}
\newtheorem{proposition}[defin]{Proposition}
\newtheorem{lemma}[defin]{Lemma}
\newtheorem{corollary}[defin]{Corollary}
\newtheorem{definition}[defin]{Definition}

\newtheorem{remark}[defin]{Remark}
\newtheorem{assumption}[defin]{Assumption}

\newtheorem*{maintheorem}{Theorem}

\newcommand{\C}{\mathbb{C}}
\newcommand{\R}{\mathbb{R}}

\newcommand{\Z}{\mathbb{Z}}

\newcommand{\Ecal}{\mathcal{E}}

\newcommand{\Mcal}{\mathcal{M}}

\newcommand{\Scal}{\mathcal{S}}

\newcommand{\lset}{\left\{}
\newcommand{\rset}{\right\}}
\newcommand{\SL}{\operatorname{SL}}

\newcommand{\ort}{\operatorname{O}}

\makeatletter
\newcommand{\bigperp}{%
  \mathop{\mathpalette\bigp@rp\relax}%
  \displaylimits
}

\newcommand{\bigp@rp}[2]{%
  \vcenter{
    \m@th\hbox{\scalebox{\ifx#1\displaystyle2.1\else1.5\fi}{$#1\perp$}}
  }%
}
\makeatother

\DeclareMathOperator{\diag}{diag}

\DeclareMathOperator{\Tr}{Tr}

\DeclareMathOperator{\Harm}{Harm}

\DeclareMathOperator{\Eis}{Eis}

\definecolor{arne}{rgb}{0.2,.7,0.9}

\definecolor{aurelio}{rgb}{0.9,.2,0.9}

\definecolor{frank}{rgb}{0.8,0.2,0.5}

\definecolor{marc}{rgb}{1,0.65,0}

\begin{document}

\title{Critical modular lattices in the Gaussian core model}

\address{A.~Joharian, F.~Vallentin,
  M.C.~Zimmermann, Department Mathematik/Informatik, Abteilung
  Mathematik, Universit\"at zu K\"oln, Weyertal~86--90, 50931 K\"oln,
  Germany}

\author{Arian Joharian}

\email{arianjoharian@icloud.com}

\author{Frank Vallentin}
\email{frank.vallentin@uni-koeln.de}

\author{Marc Christian Zimmermann}
\email{marc.christian.zimmermann@gmail.com}

\date{February 13, 2026}


\begin{abstract}
  We discuss the local analysis of Gaussian potential energy of modular lattices. 
   We present examples of $2$-modular lattices---such as the $16$-dimensional Barnes-Wall lattice---and $3$-modular lattices---such as the $12$-dimensional Coxeter-Todd lattice---that are locally universally optimal among lattices (in the sense of Cohn and Kumar).   We also provide other $2$- and $3$-modular lattices that are not locally universally optimal, or not even critical in the Gaussian core model.
\end{abstract}

\maketitle

\markboth{A. Joharian, F. Vallentin, and
  M.C. Zimmermann}{Critical modular lattices in the Gaussian core model}

\setcounter{tocdepth}{1} 
\tableofcontents

\section{Introduction}

\subsection{The Gaussian core model}

Let $L \subseteq \mathbb{R}^n$ be an $n$-dimensional lattice, i.e. a discrete subgroup of $\R^n$ of full rank.
The \emph{Gaussian potential energy} of $L$ is defined by
\[
  \mathcal{E}(\alpha,L) = \sum_{x \in L \setminus\{0\}} e^{-\alpha \|x\|^2},
\]
for $\alpha >0$.
Point configurations which interact via Gaussian potential functions $r \mapsto e^{-\alpha r^2}$ are referred to as the \emph{Gaussian core model}. 
They are natural physical systems (see Stillinger \cite{Stillinger1976a}) and they are mathematically quite general.
This embeds into the larger context of universal optimality:
Following Cohn and Kumar \cite{Cohn2007}, a point configuration is called \emph{universally optimal} if it minimizes potential energy among all configurations of the same point density for every completely monotonic function of squared distance.
It is an important observation that after fixing a point density, a point configuration is universally optimal if and only if it minimizes Gaussian potential energy for all $\alpha >0$.
This is a consequence of a theorem of Bernstein (see Widder \cite[Theorem 12b, page 161]{Widder1941a}). 
See Cohn, Kumar, Miller, Radchenko and Viazovska \cite{Cohn2019a} for further discussion.\smallskip

In the seminal paper \cite{Cohn2019a} the authors proved that the $E_8$ root lattice in dimension $8$ and the Leech lattice $\Lambda_{24}$ in dimension $24$ are universally optimal point configurations in their respective dimensions.
Further, the authors proved that these lattices are the unique minimizers among all periodic point configurations. \smallskip

In this follow-up paper to \cite{Heimendahl2023} we are interested in the \textit{local analysis} of the function $L \mapsto \mathcal{E}(\alpha, L)$ when $L$ varies in the manifold of rank $n$ lattices having point density $1$, which means that the number of lattice points per unit volume equals $1$. 
Even if considering local universal optimality restricted to lattice configurations, only few examples are known. 
Explicitly, these are, the one dimensional standard lattice $\Z$, the hexagonal lattice $A_2$ (see Montgomery \cite{Montgomery1988}), the root lattice $D_4$, $E_8$ and $\Lambda_{24}$ (see Sarnak and Str\"ombergsson \cite{Sarnak2006a}), where $\Z$, $E_8$, and $\Lambda_{24}$ of course are now known to be universally optimal among all point configurations. The hexagonal lattice is universally optimal among lattices and conjectured to be universally optimal among all point configurations. 
While $D_4$ is only proven to be locally universally optimal among lattices, it is conjectured (based on numerical simulations by Torquato and Stillinger \cite{Torquato2008} and Cohn, Kumar, and Sch\"urmann \cite{Cohn2009}) that $D_4$ is actually universally optimal among all point configurations. 
\smallskip

In \cite{Heimendahl2023} the case of even unimodular lattices was considered. 
Recall that a lattice $L$ is called \emph{unimodular} if it coincides with its dual lattice
\[
  L^\# = \{y \in \R^n : x \cdot y \in \Z \text{ for all } x \in L\},
\]
where $x \cdot y$ denotes the standard inner product of $x, y \in \R^n$.
The lattice $L$ is called \emph{even} if for every lattice vector $x \in L$ the inner product $x \cdot x$ is an even integer. 
It is well-known (see for instance Serre~\cite{Serre1973a}) that in a given dimension the number of non-isometric even unimodular lattices is finite and that they exist only in dimensions which are divisible by $8$.

One of the main goals of \cite{Heimendahl2023}, motivated by Regev and Stephens-Davidowitz~\cite{Regev2016a}, was to find a concrete example of an even unimodular lattice that is a local maximum for some energy parameter $\alpha$ and the associated Gaussian potential.
The first such example in this class of lattices can be found dimension $32$, every extremal even unimodular lattice in dimension $32$ is a local maximum for certain values of $\alpha$, a rigorous proof of this statement was worked out specifically for $\alpha = \pi$.
It was proved that $E_8 \perp E_8$ in dimension $16$ and every Niemeier lattice in dimension $24$ with a decomposable root system are saddle points if $\alpha$ is large enough.
While all even unimodular lattices in dimensions up to $24$ turned out to be critical points in the Gaussian core model, there are non-critical even unimodular lattices in dimension $32$. \smallskip

The aim of this paper is to extend the local analysis to a larger class of lattices, including dimensions in which no even unimodular lattice can exist. 
For this we consider the class of (extremal) modular lattices.
This is in part motivated by an easy observation, which was already mentioned in Sarnak and Str\"ombergsson \cite{Sarnak2006a}:
A lattice (of point density $1$) cannot be universally optimal if the theta series of $L$ and $L^\#$ are different.
It seems noteworthy to mention that $A_2$ and $D_4$, which are locally universally optimal among lattices, are actually $2$-modular lattices, that is $A_2 \cong \sqrt{2}A_2^\#$ and $D_4 \cong \sqrt{2}D_4^\#$.
From the general definition of an $\ell$-modular lattice below, it becomes clear that $\ell$-modular lattices satisfy this necessary condition and therefore seem good candidates to investigate if one is interested in universal optimality.

\subsection{Modular lattices in the Gaussian core model}

Modular lattices generalize unimodular lattices.
This generalization is closely related to the \emph{theta series} of a lattice $L$, which is a function of a variable $\tau$ in the complex upper half plane given by
\[
  \Theta_L(\tau) = \sum_{v \in L} e^{\pi i \tau \|v\|^2}.
\]
If $L$ is an even lattice, $\Theta_L$ can be rewritten in terms of the cardinalities $a_m = |L(2m)|$ of the \emph{shells} $L(2m) = \lset v \in L : \|v\|^2 = 2m \rset$ for $m \in \Z_{\geq 0}$ and the modulus $q = e^{2\pi i \tau}$ as a Fourier-series
\[
  \Theta_L(q) = \sum_{m \geq 0} a_m q^m.
\]
It is well known that if $L$ is even unimodular, then $\Theta_L$ is a modular form for the full modular group $\SL_2(\Z)$.\smallskip

We now move to a larger class of lattices, for which $\Theta_L$ has similar desirable properties.
Let $\ell > 0$.
A \emph{similarity} $\sigma: \R^n \rightarrow \R^n$ of \emph{norm} $\ell$ is a linear map satisfying
\[
  \sigma(v) \cdot \sigma(w) = \ell v \cdot w 
\]
for all $v,w \in \R^n$. 

\begin{definition}
  A lattice is called $\ell$-modular if there exists a \emph{similarity} of \emph{norm} (or \emph{similarity factor}) $\ell$ such that 
  \[
    L = \sigma(L^\#),
  \]
  or, equivalently,
  \[
    L \cong \sqrt{\ell} L^\#.
  \] 
\end{definition}

As first observed by Quebbemann in \cite{Quebbemann1995} (and extended in \cite{Quebbemann1997}) the theta series of an even $\ell$-modular lattice of dimension $n=2k$ is a modular form of weight $k$ for the \emph{Fricke group} $\Gamma_*(\ell)$ and a suitably chosen character.

This connection to modular forms implies an upper bound on the minimal norm of a nonzero lattice vector. 
Modular lattices for which the minimum norm attains this bound are called \emph{extremal}.
For a definition of $\Gamma_*(\ell)$ and details on the character and extremality we refer forward to Section \ref{sec:theta-series}.

We will investigate the local behavior of Gaussian potential energy of modular lattices in terms of the theta series of $L$ using the relation
\[
  \mathcal{E}(\alpha,L) = \sum_{x \in L \setminus\{0\}} e^{-\alpha \|x\|^2} = \Theta_L(\alpha i / \pi) -1.
\]

Note that an $\ell$-modular lattice $L$ is not actually a point on the manifold of lattices of point density $1$, as the point density of an $\ell$-modular lattice is $\ell^{n/4}$ since its determinant is $\ell^{n/2}$.
So we need to rescale $L$ to have point density $1$ and obtain the rescaled lattice $\ell^{-n/4} L$.
In this case, we have the relation
\begin{equation}\label{eq:rescaling:energy}
  \Ecal(\alpha,L) = \Ecal(\ell^{n/2}\alpha,\ell^{-n/4} L).
\end{equation}
While this is important to keep in mind, when comparing actual values of $\Ecal$, it does not matter for the validity of qualitative statements of the form:
There exists some $\alpha$ such that $L$ is a local maximum/minimum or saddle point for energy with respect to the Gaussian potential $r \mapsto e^{-\alpha r^2}$.
As the present paper deals precisely with questions of that nature, we will work with $\ell$-modular lattices directly and not with their rescaled counterparts.

\subsection{Structure of the paper}

We will collect some technical preliminaries in Section \ref{sec:prelim}, including some background on spherical designs, harmonic polynomials, theta series (with spherical coefficients) of modular lattices, and extremal modular lattices.

In Section \ref{sec:Strategy} we explain the basic strategy we employ in our computations. 
While this is similar to the treatment in \cite{Heimendahl2023}, it deviates in some crucial points and certain new aspects in the analysis of the eigenvalues of the Hessian of energy emerge.

The application of this strategy requires bounds on the coefficients of modular forms, which we obtain separately for Eisenstein series and cusp forms. This is the content of Section \ref{sec:bounds}. 

Finally, in Section \ref{sec:results} we bring all of the preparations together and study some examples of extremal $2$ and $3$-modular lattices in detail. 
In principle, one can extend these computations, but we chose to do this only in a limited number of cases, each of which is either concentrating on a prominent lattice or highlights a certain kind of behavior.

\subsection{Main results}

We can summarize our results as follows.
The main technical contribution of the paper is a technique that allows us to compute the eigenvalues of the Hessian of energy with respect to Gaussian potentials, if $L$ is an even $\ell$-modular lattice. This is the content of Theorem \ref{thm:maintheorem}.

Building on this, if $L$ is an extremal even $\ell$-modular lattice, where $\ell$ is prime and $1 + \ell$ divides $24$, we can determine the asymptotic behavior of $L$ in the Gaussian core model, as $\alpha \rightarrow \infty$, under the assumptions made for Theorem \ref{thm:maintheorem}. 
This is the content of Theorem \ref{thm:large-alpha}, where it turns out that, asymptotically, $L$ is either a local minimum or a saddle point, depending on its minimal vectors.

We apply the aforementioned techniques to $2$ and $3$-modular lattices in dimensions up to $20$ and $18$ respectively; these are the dimensions in which the classification of all extremal $2$- and $3$-modular lattices has been completed.
We collect numerical and rigorous results on all of these cases in Tables \ref{tab:numerics:modular:lattices:2} and \ref{tab:numerics:modular:lattices:3}, see Section \ref{sec:numerical} for more information on these tables.

Specifically, while the tables contain a good amount of data which is only numerical evidence, this can be made rigorous.
We illustrate this with an in depth example in Section \ref{sec:2design}:
Here, we show how to deal with lattices for which the shells are spherical $2$-designs but not spherical $4$-designs. 
Then a more careful analysis is necessary, and this becomes rather technical if an associated space of cusp forms is of dimension at least $2$.
This phenomenon occurs for extremal $\ell$-modular lattices, as opposed to the previous study of even unimodular lattices up to dimension $32$ in \cite{Heimendahl2023}.
To illustrate the method, we apply it to one (out of six) extremal $3$-modular lattice in dimension $16$.
Numerical evidence suggests that this lattice is a saddle point for all $\alpha$.
We provide a rigorous proof of this for $\alpha = 1$, to support this observation, and to illustrate the method.
This is collected in Proposition \ref{prop:L16}.

\smallskip

In addition to the above, we discuss some cases where the shells of the lattices form stronger designs, namely at least $4$-designs.
This includes the $3$-modular Coxeter-Todd lattice in dimension $12$ and the $2$-modular Barnes-Wall lattice in dimension $16$.

The main idea is to utilize that these lattices possess a lot of structure that simplifies the computation of the eigenvalues of the Hessian of $\Ecal(\alpha,L)$ significantly; this comes from the fact that the shells of these lattices form spherical $4$-designs.
Recall that if the first shell, i.e. the shell associated to minimal vectors of the lattice, forms a spherical $4$-design, then the lattice is called \emph{strongly perfect}, a notion by Venkov~\cite{Venkov2001a}. It is known that strongly perfect lattices are \emph{extreme} in the geometric sense, i.e. they are strict local maxima of the sphere packing density function for lattices.
By a classical result of Voronoi \cite{Voronoi1907} a lattice is extreme if and only if it is perfect and eutactic.
Here a lattice $L$ is \emph{perfect} if the span of the rank one matrices $x x^{\sf{T}}$, where $x$ runs through the set of minimal vectors of $L$, is equal to the space of symmetric matrices; $L$ is \emph{eutactic} if the identity matrix is in the interior of the convex cone spanned by the rank one matrices $x x^{\sf{T}}$ from before.

We show that both lattices are locally universally optimal (see Propositions \ref{prop:CT} and \ref{prop:BW}).
This seems plausible as both lattices are the best known sphere packings in their respective dimension and the energy minimization problem converges to the sphere packing problem for $\alpha \rightarrow \infty$. 

A direct extension of the proof of local universal optimality among lattices of these two well-known lattices, results in a criterion that asserts local universal optimality for other extremal $\ell$-modular lattices.

\begin{maintheorem}
  Let $L$ be an extremal $\ell$-modular lattice for which all shells are $4$-designs. Then, $L$ is locally universally optimal among lattices if 
  $\ell \geq  3$, or if $\ell = 2$ and $\dim(L) \in \lset 4, 16, 20, 32, 48 \rset$.
\end{maintheorem}

This theorem (Theorem~\ref{thm:ellmodular:loc}) is proven in Section \ref{sec:loc:univ} and Table \ref{tab:loc:univ} collects all extremal $\ell$-modular lattices up to dimension $20$.

\smallskip

We complement the analysis in the presence of $2$-designs with a short proof (see Section \ref{sec:noncritical}) that there exists an extremal $2$-modular lattice in dimension $12$ which is not a universal critical point in the Gaussian core model.

What is different here is that the $32$-dimensional example came from a non-extremal lattice, so a lattice which cannot be optimal for sphere packing, while the new example is extremal.
However, all $2$-modular lattices in dimension $12$ fail to be extreme (in the geometric sense), as they are neither perfect nor eutactic, which we checked by computer.
In any case it is clear that the new $2$-modular example is not an optimal sphere packing, as it packs less dense than the $3$-modular Coxeter-Todd lattice.

We briefly mention that being non-critical is not a property of all $2$-modular lattices in dimension $12$.
Up to isometry there are $3$ such, and $2$ of them are non-critical, while one of them is. 
See the discussion in Section \ref{sec:noncritical} and Section \ref{sec:numerical}.

\section{Preliminaries} \label{sec:prelim}

\subsection{Spherical designs}
We reproduce this discussion from \cite{Heimendahl2023}, to keep the manuscript self-contained.

A finite set $X$ on the sphere of radius $r$ in $\R^n$ denoted by $S^{n-1}(r)$ is called a spherical $t$-design if
\[
  \int_{S^{n-1}(r)} p(x) \, dx = \frac{1}{|X|} \sum_{x \in X} p(x)
\]
holds for every polynomial $p$ of degree up to $t$. Here we integrate with respect to the rotationally invariant probability measure on $S^{n-1}(r)$.

If $X$ forms a spherical $2$-design, then 
\begin{equation}
\label{eq:spherical-2-design}
  \sum_{x \in X} xx^{\sf T} = \frac{r^2 |X|}{n} I_n
\end{equation}
holds, where $I_n$ denotes the identity matrix with $n$ rows/columns.

A polynomial $p \in \R[x_1, \ldots, x_n]$ is called harmonic if it vanishes under the Laplace operator
\[
  \Delta p = \sum_{i=1}^n \frac{\partial^2 p}{\partial x_i^2} = 0.
\]
We denote the space of homogeneous harmonic polynomials of degree $k$ by $\Harm_k$. One can uniquely decompose every homogeneous polynomial $p$ of even degree $k$
\begin{equation}
\label{eq:harmonic-decomposition}
p(x) = p_k(x) + \|x\|^2 p_{k-2}(x) + \|x\|^4 p_{k-4}(x) + \cdots +
\|x\|^k p_0(x)
\end{equation}
with $p_d \in \Harm_d$ and $d = 0, 2, \ldots, k$.

We can characterize that $X$ is a spherical $t$-design by saying that the sum $\sum_{x \in X} p(x)$ vanishes for all homogeneous harmonic polynomials $p$ of degree $1, \ldots, t$.

\smallskip

In the following we shall need the following technical lemma.

\begin{lemma}[Lemma 2.1 in \cite{Heimendahl2023}]
\label{lemma:polynomial-p-H}
Let $H$ be a symmetric matrix with trace zero. The homogeneous polynomial
\[
p_H(x) = (x^{\sf T} H x)^2 = H[x]^2
\]
of degree four decomposes as in \eqref{eq:harmonic-decomposition}
\[
  p_H(x) = p_{H,4}(x) + \|x\|^2 p_{H,2}(x) + \|x\|^4 p_{H,0}(x)
\]
with $p_{H,d} \in \Harm_d$ and
\[
  p_{H,4}(x) = p_H(x) - \|x\|^2 \frac{4}{4+n} H^2[x] + \|x\|^4 \frac{2}{(4+n)(2+n)}
  \Tr H^2
\]
and
\[
  p_{H,0}(x) = \frac{2}{(2+n)n} \Tr H^2.
\]
\end{lemma}

\subsection{Theta series as modular forms}\label{sec:theta-series}

As noted before, if $L$ is an even $\ell$-modular lattice, its theta series can be rewritten as a Fourier-series in $q = e^{2\pi i \tau}$, namely
\[
  \Theta_L(q) = \sum_{m \geq 0} a_m q^m
\]
where for $m \in \Z_{\geq 0}$ the coefficient $a_m = |L(2m)|$ is the cardinality of the shell $L(2m) = \lset v \in L : \|v\|^2 = 2m \rset$.

If $L$ is an even unimodular lattice of dimension $n=2k$, the theta series $\Theta_L$ is a modular form of weight $k$ for the full modular group $\SL_2(\Z)$.

This is no longer true for arbitrary even $\ell$-modular lattices.
So we consider the congruence subgroup
\[
  \Gamma_0(\ell) = \lset \begin{pmatrix} a & b \\ c & d \end{pmatrix} \in \SL_2(\Z) : c = 0 \bmod \ell \rset  
\]
and the character $\chi$ on $\Gamma_0(\ell)$ given in terms of the Kronecker symbol
\begin{equation} \label{eq:chi:1}
  \chi(A) = \left( \frac{(-\ell)^k}{d} \right) \quad \text{for } A = \begin{pmatrix} a & b \\ c & d \end{pmatrix} \in \Gamma_0(\ell).
\end{equation}
Then $\Theta_L$ is a modular form for the group $\Gamma_0(\ell)$ and $\chi$.
Moreover, if $p$ is a harmonic polynomial of degree $d$ this is also true for the \emph{theta series with spherical coefficients}
\[
  \Theta_{L,p}(q) = \sum_{x \in L} p(x) q^{\tfrac{1}{2}\|x\|^2}.
\]
To be more precise, it turns out that $\Theta_{L,p}$ is a cusp form of weight $n/2 + \deg(p)$ if $\deg(p) \geq 1$.
For a proof of these statements we refer to \cite[Theorem 3.2]{Ebeling1994a}.

However, for a direct application the space of modular forms $\Mcal(\Gamma_0(\ell),\chi)$ is often too large and better replaced by a smaller space with more (and easier to use) structure.

We give some detail for the case $\ell$ being a prime number.
Then the normalizer of $\Gamma_0(\ell)$ in $\SL_2(\R)$ contains a special involution $W_\ell$, the so-called \emph{Fricke involution}, given by
\[
  W_\ell = \begin{pmatrix} 0 & 1/\sqrt{\ell} \\ -\sqrt{\ell} & 0 \end{pmatrix}.
\]
With this we can extend $\Gamma_0(\ell)$ to the \emph{Fricke group}
\[
  \Gamma_*(\ell) = \Gamma_0(\ell) \cup \Gamma_0(\ell) W_\ell.
\]
To work with theta series of $\ell$-modular lattices of dimension $n = 2k$ we can now extend the character $\chi$ from \eqref{eq:chi:1} to obtain a character on $\Gamma_*(\ell)$ by setting
\[
  \chi(W_\ell) = i^k.
\]
Then if $L$ is an even $\ell$-modular lattice of dimension $n = 2k$ we finally obtain that $\Theta_L$ is a modular form of weight $k$ for $\Gamma_*(\ell)$ and the character $\chi$, in notation $\Theta_L \in \Mcal_k(\Gamma_*(\ell),\chi)$.

What is especially nice about $\Mcal_k(\Gamma_*(\ell),\chi)$ and its associated subspace of cusp forms $\Scal_k(\Gamma_*(\ell),\chi)$ is the following.
For $\ell$ such that $1+\ell$ divides $24$ we have an identification of algebras
\begin{equation} \label{eq:starspace:polynomialalgebra}
  \Mcal(\Gamma_*(\ell),\chi) = \bigoplus_{k \geq 0} \Mcal_k(\Gamma_*(\ell),\chi) = \C[\Theta,\Delta_{\ell}], 
\end{equation}
where $\Theta$ is the theta series of any even $\ell$-modular lattice of lowest possible dimension and $\Delta_{\ell}(\tau) = \left(\eta(\tau) \eta(\ell \tau) \right)^{24/(1 + \ell)}$, where $\eta$ is the Dedekind eta function.
Here, the list of levels $\ell$ with $\ell +1 \div 24$ is given by the primes $2,3,5,7,11,23$. Note that a similar algebra decomposition exists for a slightly larger class of levels, in particular the full level $\ell =1$, and the composite levels $6,14,15$ (see \cite{Quebbemann1997}).

Another important property of $\Mcal_k(\Gamma_*(\ell),\chi)$ for $\ell$ such that $1+\ell$ divides $24$ is that it is a space of modular forms for which \emph{extremality is definable}.
That means that the projection 
\begin{equation} \label{eq:extremality}
  \Mcal_k(\Gamma_*(\ell),\chi) \rightarrow \C^d; f = \sum_{m \geq 0} a_m q^m \mapsto (a_0,\ldots,a_{d-1})
\end{equation}
onto the first $d = \dim(\Mcal(\Gamma_*(\ell),\chi))$ coefficients of the Fourier expansion of an element $f \in \Mcal(\Gamma_*(\ell),\chi)$ is injective.
This means that for $\Mcal_k(\Gamma_*(\ell),\chi)$ there exists a unique element $F$ with Fourier expansion 
\[
  F = 1 + \sum_{m \geq d} a_m q^m.
\]
This element is called the \emph{extremal modular form} in $\Mcal(\Gamma_*(\ell),\chi)$.

If $L$ is an even $\ell$-modular lattice of dimension $n = 2k$ such that $\Theta_L$ is the extremal modular form in $\Mcal_k(\Gamma_*(\ell),\chi)$, then $L$ is called an \emph{extremal $\ell$-modular lattice}.
Note that such a lattice, if it exists, which is still an open question in many cases, needs not be unique. 
For an overview of what is known we refer to \cite{Nebe}.

\subsection{Strongly modular lattices}
The considerations of the previous section apply to a slightly more general class of lattices, so called strongly modular lattices.
While we do not conduct explicit computations for such lattices in the remainder of the paper, the techniques introduced here after carry over to this class.

For the definition consider the quotient $L^\# / L$, the discriminant group of $L$, and the exponent $\ell$ of this group, i.e.\ the smallest positive integer $\ell$ such that $\ell(v+L) = L$ for all $v \in L^\#$.
A lattice $L$ is called \emph{strongly $\ell$-modular} if for all divisors $m$ of $\ell$, for which $m,\ell/m$ are co-prime, we have
\[
  L \cong \sqrt{m} \left( \frac{1}{m}L \cap L^\# \right).
\] 
Note that if $\ell$ is prime, this is the same as $L$ being $\ell$-modular.

The theta series of strongly $\ell$-modular lattices a modular form invariant under the character $\chi$ from above and an eigenform for all Atkin-Lehner involutions associated to divisors $m$ of $\ell$, for which $m,\ell/m$ are co-prime.
If $\ell$ is a prime, there is one such involution, the Fricke involution.

For $\ell$ such that the sum of divisors $\sigma_1(\ell)$ of $\ell$ divides $24$, Quebbemann \cite{Quebbemann1997} showed that the resulting space of modular forms has similar properties to $\Mcal_k(\Gamma_*(\ell),\chi)$ as above, in particular there is an analogue of \eqref{eq:starspace:polynomialalgebra}.
For an in depth discussion of strongly modular lattices and explicit information on the associated space of modular forms we refer to \cite{Quebbemann1997} and \cite{Scharlau1999}.

\subsection{Extremal modular lattices and spherical designs} \label{sec:spherical:shells}

There are two known, general ways to establish a minimum design strength on the shells of extremal $\ell$-modular lattices.

Firstly, Bachoc and Venkov \cite{Bachoc2001} provide such bounds specifically for  extremal $\ell$-modular lattices, when $\ell$ prime and $1+\ell$ divides $24$.
This is a generalization of Venkov's classical result regarding unimodular lattices \cite{Venkov2001a}, it can be found as Corollary 4.1 in \cite{Bachoc2001}.
We mainly use this for extremal $2$ and $3$-modular lattices.

In these cases their results read as follows:
If $L$ is an extremal $2$-modular lattice and $L(2m)$ is non-empty, then $L(2m)$ is a spherical $t$-design with
\begin{equation} \label{eq:shell:designs:2}
  t = 
  \begin{cases}
    7 & \text{if } n = 0 \mod 16,\\
    5 & \text{if } n = 4 \mod 16,\\
    3 & \text{if } n = 8 \mod 16.
  \end{cases}
\end{equation}
If $L$ is an extremal $3$-modular lattice and $L(2m)$ is non-empty, then $L(2m)$ is a spherical $t$-design with
\begin{equation} \label{eq:shell:designs:3}
  t = 
  \begin{cases}
    5 & \text{if } n = 0,2 \bmod 12,\\
    3 & \text{if } n = 4,6 \bmod 12.
  \end{cases}
\end{equation}

In all dimensions up to $18$ ($3$-modular) and $20$ ($2$-modular) all such lattices have been classified and additional data on them is readily available through the ``Catalogue of lattices'' \cite{Nebe}. We will collect some of this information in Section \ref{sec:2-3-modular}.

\smallskip

A second way, which works for arbitrary lattices, is based on a result of Goethals and Seidel \cite{Goethals1981}; we refer to this paper for all information exceeding the brief summary below.
Let $L$ be a lattice, and consider the \emph{harmonic Molien series} of $\ort(L)$, the orthogonal group of $L$, which contains information about harmonic invariants of $\ort(L)$:
The $i$th coefficient of the harmonic Molien series of $\ort(L)$ equals the dimension of the space of harmonic invariants of degree $i$ of $\ort(L)$.
Therefore, one can check that there are no $\ort(L)$-invariant harmonic polynomials of degrees $1,\ldots,t$ for the isometry group $\ort(L)$ of $L$, by computing the series and checking its coefficients; in practice we use the computer algebra system MAGMA \cite{MAGMA} to do this for a specific lattice.
This is helpful, since Theorem 3.12 in \cite{Goethals1981} asserts that all $\ort(L)$-orbits are spherical $t$-designs if and only if there are no harmonic invariants of degree $1,\ldots,t$, this applies, in particular, to all shells of $L$, as they clearly are unions of orbits of $\ort(L)$.

Note that this only provides a lower bound on the design strength of the shells of the lattice, as it considers \emph{all} $\ort(L)$-orbits, in fact one can find an example for this among the $2$-modular lattices in dimension $20$. 
There are $3$ such, and by the result of Bachoc and Venkov, all of their shells are at least $5$-designs. 
For two of these lattices this is also a consequence of the approach using the harmonic Molien series, but for one of them (the second entry in the file ``\texttt{2\_dim20.dat}'' in \cite{Nebe}) the latter approach only yields a minimal strength of $3$, as there are harmonic invariants of degree $4$ for this specific lattice and its orthogonal group.

\section{Strategy}
\label{sec:Strategy}

We reproduce some of the discussion from \cite{Heimendahl2023}, but we will have to extend the approach to be able to handle some cases of lattices whose shells are not $4$-designs.

Our main goal is to evaluate the gradient and Hessian of $L \mapsto \mathcal{E}(\alpha, L)$ on the manifold of lattices of fixed point density; recall the relation~\eqref{eq:rescaling:energy} and the associated discussion. 

The gradient and the Hessian of $\mathcal{E}(\alpha,L)$ at $L$ were computed by Coulangeon and Sch\"urmann \cite[Lemma 3.2]{Coulangeon2012a}, and we quickly collect their result here. 
Let $H$ be a symmetric matrix having trace zero (lying in the tangent space of the identity matrix). 
We use the notation $H[x] = x^{\sf T} H x$, and we equip the space of symmetric matrices $\mathcal{S}^n$ with the inner product $\langle A, B \rangle = \Tr(AB)$, where $A, B \in \mathcal{S}^n$. 
The gradient is given by
\begin{equation}
\label{eq:gradient}
\langle  \nabla  \mathcal{E}(\alpha,L), H \rangle = -\alpha \sum_{x \in L \setminus \{0\}} H[x] e^{-\alpha \|x\|^2},
\end{equation}
while the Hessian is the quadratic form
\begin{equation}
\label{eq:Hessian}
\nabla^2 \mathcal{E}(\alpha,L)[H] = \alpha \sum_{x \in L\setminus\{0\}}
e^{-\alpha\|x\|^2} \left(\frac{\alpha}{2} H[x]^2 - \frac{1}{2}H^2[x]\right).
\end{equation}

In this section we will show that for a (modular) lattice $L$ the gradient \eqref{eq:gradient} vanishes if all shells form spherical $2$-designs, and we will compute the eigenvalues of the Hessian \eqref{eq:Hessian} in terms of the coefficients of the theta series of $L$.
This is summarized in Lemma \ref{lem:Hessian:4design} if all shells form spherical $4$-designs and more generally in Theorem \ref{thm:maintheorem} if all shells form at least spherical $2$-designs.

\subsection{The gradient in the presence of $2$-designs}

Now a sufficient condition for $L$ being a critical point is that all shells of $L$ form spherical $2$-designs. 
Indeed, we group the sum in~\eqref{eq:gradient} according to shells, giving
\[
\langle  \nabla  \mathcal{E}(\alpha,L), H \rangle =  -\alpha \sum_{r > 0} e^{-\alpha r^2} \sum_{x \in L(r^2)} H[x].
\]
Then for $r > 0$ every summand
\[
  \sum_{x \in L(r^2)} H[x]
  = \left\langle H, \sum_{x \in L(r^2)} x x^{\sf T}
   \right\rangle =  \frac{r^2 |X|}{n} \Tr(H) = 0
\]
vanishes because of~\eqref{eq:spherical-2-design} and because $H$ is traceless. 
Hence, $L$ is critical.

\subsection{The Hessian in the presence of $4$-designs}
We rewrite the expression \eqref{eq:Hessian} by grouping the sum according to shells, which gives
\begin{equation}
  \nabla^2 \mathcal{E}(\alpha,L)[H] = \alpha \sum_{r > 0} e^{-\alpha
    r^2} \sum_{x \in L(r^2)} \left(\frac{\alpha}{2} H[x]^2 - \frac{1}{2}H^2[x]\right).
\end{equation}
We now strive to evaluate the two sums
\begin{equation}
  \label{eq:two-sums}
  \sum_{x \in L(r^2)} H[x]^2 \quad \text{and} \quad \sum_{x \in
    L(r^2)} H^2[x].
\end{equation}
The second sum is easy to compute when $L(r^2)$ forms a spherical $2$-design. 
In this case we have by~\eqref{eq:spherical-2-design}
\begin{equation}
\label{eq:spherical-2-design-H-squared}
  \sum_{x \in L(r^2)} H^2[x] = \left\langle H^2, \sum_{x \in L(r^2)}
      xx^{\sf T} \right\rangle  = \langle H^2,   \frac{r^2
      |L(r^2)|}{n} I_n \rangle =
  \frac{r^2 |L(r^2)|}{n} \Tr H^2.
\end{equation}
The first sum is only easy to compute when $L(r^2)$ forms a spherical $4$-design. Then (see \cite[Proposition 2.2]{Coulangeon2006a} for the computation)
\begin{equation}\label{eq:spherical-4-design-quadratic-form}
 \sum_{x \in L(r^2)} H[x]^2 = \frac{r^4 |L(r^2)|}{n(n+2)} 2 \Tr H^2.
\end{equation}
Together, when all shells form spherical $4$-designs, the Hessian~\eqref{eq:Hessian} simplifies to
\begin{equation}
  \label{eq:Coulangeon-Hessian}
  \nabla^2 \mathcal{E}(\alpha,L)[H] = \frac{\Tr H^2}{n(n+2)}
  \sum_{r > 0} |L(r^2)| \alpha r^2 \left(\alpha r^2 - (n/2 + 1)\right) e^{-\alpha r^2}.
\end{equation}
Therefore, every $H$ with Frobenius norm $\langle H, H \rangle = \Tr H^2 = 1$ is mapped to the same value, which, by choosing an orthonormal basis of the space of traceless symmetric matrices, implies that all the eigenvalues of the Hessian coincide.
\begin{lemma}\label{lem:Hessian:4design}
  Let $L$ be lattice, such that all shells form spherical $4$-designs. Then all the eigenvalues of the Hessian coincide and are given by
   \begin{equation} \label{eq:Hessian:4design}
      \nabla^2 \Ecal(\alpha,L)[H] = \frac{1}{n(n+2)} \sum_{r>0} \vert L(r^2) \vert \alpha r^2 (\alpha r^2 - (n/2 + 1)) e^{-\alpha r^2}.
  \end{equation}
\end{lemma}

\subsection{The Hessian in the presence of $2$-designs}

The case when all shells form spherical $2$-designs but not (necessarily) spherical $4$-designs requires substantially more work. 
It is here where we first exploit the relation of modular lattices to modular forms for congruence subgroups.
Our investigation involves the quadratic forms 
\begin{equation}
\label{eq:crucial-quadratic-form}
  Q_{k}[H] = \sum_{x \in L(k)} H[x]^2
\end{equation}
defined for $k$ such that $L(k)$ is non-empty.

We start by investigating the first sum $\sum_{x \in L(r^2)} H[x]^2$ in~\eqref{eq:two-sums} in more detail. 
We decompose the polynomial $p_H(x) = H[x]^2$ into its harmonic components as in Lemma~\ref{lemma:polynomial-p-H} and get
  \[
    \sum_{x \in L(r^2)} p_H(x) = \sum_{x \in L(r^2)} p_{H,4}(x) + r^2 \sum_{x \in L(r^2)} p_{H,2}(x) + r^4 \sum_{x \in L(r^2)} p_{H,0}(x).
\]
Here the first sum equals
\[
\sum_{x \in L(r^2)} p_{H,4}(x) =  \sum_{x \in L(r^2)} H[x]^2 -
r^4 \frac{2}{(2+n)n} |L(r^2)| \Tr H^2,
\]
where we used Lemma~\ref{lemma:polynomial-p-H} and~\eqref{eq:spherical-2-design-H-squared}. 
The second sum vanishes
because $L(r^2)$ is a spherical $2$-design and the third summand equals
\[
  r^4 \sum_{x \in L(r^2)} p_{H,0}(x) = r^4  \frac{2}{(2+n)n} |L(r^2)| \Tr H^2.
\]

To understand $H[x]^2$, we analyze $p_{H,4}$. 
Since $p_{H,4}$ is a harmonic polynomial of degree 4, the theta series $\Theta_{L,p_{H,4}}$ has spherical coefficients and is therefore a cusp form of weight $n/2 + 4$ for the group $\Gamma_0(\ell)$ and the character $\chi$ from Section \ref{sec:theta-series}. 
We can use that fact to determine $\sum_{x \in L(r^2)} p_{H,4}(x)$ explicitly. \\
Let $e$ denote the dimension of the cuspidal space $\Scal_{n/2+4}(\Gamma_0(\ell),\chi)$ and let $C_1,\dots,C_e$ be a basis.
Consider the $q$-expansion of $C_i(\tau) = \sum_{m = 0}^{\infty} c_{i,m}q^m$. 
Since $C_1,\dots,C_e$ form a basis of the cuspidal space, there exist coefficients $b_1,\dots,b_e \in \R$, such that 
\begin{align}\label{eq:thetaLph4}
    \Theta_{L,p_{H,4}}(\tau) = \sum_{r>0} \left(\sum_{x \in L(r^2)} p_{H,4}(x) \right)q^{\frac{1}{2}r^2} = \sum_{i = 1}^{e} b_i \sum_{m = 0}^{\infty} c_{i,m}q^m.
\end{align}

To simplify the investigation of the above expression we now make the following assumption:\smallskip

\begin{assumption}
  \label{assumption}
$L$ is an $n$-dimensional $\ell$-modular lattice and $\Scal_{n/2 +4}(\Gamma_0(\ell),\chi)$ has a basis $C_1,\ldots,C_e$ in reduced row echelon form.
\end{assumption}

Under this assumption the coefficients $c_{i,m}$ of $C_i$ satisfy
\begin{equation} \label{eq:row:echelon}
    c_{i,m} = \begin{cases}
        1, \quad \text{ if } i = m, \\
        0, \quad \text{ if } i \neq m, m \in \{1,\dots,e\}\\
        *, \quad \text{ else.}
    \end{cases}
\end{equation}
We will briefly discuss in how far this assumption is a restriction of the full problem in Section \ref{sec:asymptotic}, in any case the assumption is met in all the cases we are mainly interested in.
This can be seen by direct computation; for $\ell = 2$ also compare the basis $F_{k,m}$ in \cite{Jenkins2015}.

Now we can compute the $b_i$ explicitly by comparing coefficients in both expansions above.
Equating the first $e$ coefficients yields
\begin{align}\label{eq:biconstants}
    b_i = \sum_{x\in L(2i)} H[x]^2 - \frac{8 i^2}{n(n+2)} \vert L(2i) \vert \Tr(H^2).
\end{align}
Therefore, the coefficients $b_i$ rely on the quadratic forms 
\begin{align}\label{eq:quadraticforms}
    Q_{2m}[H] =  \sum_{x\in L(2m)} H[x]^2, \quad m \in \{1,\dots,e\}.
\end{align}

In general, we obtain a formula for $r^2 = 2m$, namely
\begin{align}\label{eq:formelsummehx2}
    \sum_{x \in L(2m)} H[x]^2 = \sum_{i = 1}^{e} b_ic_{i,m} + \frac{8m^2}{n(n+2)} \vert L(2m) \vert \Tr(H^2).
\end{align}
We now use this information to find the eigenvalues of the Hessian \eqref{eq:Hessian} in terms of the eigenvalues of the quadratic forms $Q_{0},\ldots,Q_{2e}$.
Given $Q_{2m}$ as above its eigenvalues are the eigenvalues of the associated bilinear form $B_{Q_{2m}}: \mathcal{S}^n \times \mathcal{S}^n \rightarrow \R$ which is given by
\begin{align} \label{eq:bil-form-b_Q}
    B_{Q_{2m}}(G,H) = \sum_{x \in L(2m)} G[x]H[x].
\end{align}
We fix an orthogonal basis $(H_i)$ of the space of the symmetric matrices with trace zero and compute the eigenvalues of the Gram matrix of $B_{Q_{2m}}$.
Now if $H$ is an eigenvector of the Gram matrix for the eigenvalue $\lambda$, we have 
\begin{align*}
    \sum_{x\in L(2m)} H[x]^2 = \lambda \Tr(H^2).
\end{align*}
Using such $H$ for $Q_{2i}$ in \eqref{eq:biconstants} gives a formula for $b_i$:
\begin{align*}
    b_i = \left(\lambda - \frac{8 i^2}{n(n+2)} \vert L(2i) \vert\right) \Tr(H^2).
\end{align*}

If $Q_{2},\ldots,Q_{2e}$ are \emph{simultaneously diagonalizable}, we can conclude that for each shared eigenvector $H$, with eigenvalue $\lambda_i$ for $Q_{2i}$, we have
\begin{align*}
  \sum_{x \in L(2m)} H[x]^2 = &\sum_{i = 1}^{e} \left(c_{i,m}\left(\lambda_i - \frac{8 i^2}{n(n+2)} \vert L(2i) \vert\right) \Tr(H^2)\right) \\
  &+ 4m^2 \frac{2}{(2+n)n} \vert L(2m) \vert \Tr(H^2).
\end{align*}

Combining all the results, we can obtain the following theorem. 

\begin{theorem}\label{thm:maintheorem}
    Let $L$ be an even $\ell$-modular lattice, such that all shells of $L$ form spherical $2$-designs.
    Let 
    \begin{align*}
        \Theta_L(\tau) = \sum_{m=0}^{\infty}a_mq^m \quad \text{ with } a_m = \vert L(2m)\vert
    \end{align*}
    be the theta series of $L$.
    Assume that Assumption~\ref{assumption} holds and 
that $Q_{2},\dots,Q_{2e}$ are simultaneously diagonalizable with common eigenspaces $E_1,\dots,E_s$ and associated eigenvalues $\lambda_{i,1},\dots,\lambda_{i,s}$ for $Q_{2i}$.
    Then the eigenvalues of the Hessian $\nabla^2 \mathcal{E}(\alpha,L)$ are given by 
    \begin{equation} \label{eq:hessian:eigenvalues}
        \begin{aligned}
          &\frac{1}{n(n+2)}\sum_{m = 1}^{\infty} \left(\sum_{i = 1}^{e} \left(c_{i,m}\frac{\alpha^2}{2}\left(\lambda_{i,k}n(n+2)- 8 i^2 a_i \right)\right)\right) e^{-2\alpha m} \\
        &+ \frac{1}{n(n+2)}\sum_{m = 1}^{\infty} (a_m2\alpha m(2\alpha m - (n/2 + 1))) e^{-2\alpha m}
        \end{aligned}
      \end{equation}
    for $k \in \{ 1,\ldots,s \}$.
\end{theorem}

Note that this theorem also includes the case when all shells of $L$
form spherical $4$-designs like in~\eqref{eq:Hessian:4design}
because of~\eqref{eq:spherical-4-design-quadratic-form}. 

\begin{remark}
  The assumptions in the above theorem are necessary to the best of our current knowledge. 
  To us, it is an interesting question when the bases in row echelon form exist (see also the following section) and whether the quadratic forms $Q_1,\ldots,Q_{2e}$ are simultaneously diagonalizable in general. 
  For the explicit examples in Section \ref{sec:results} these conditions can be checked, and the theorem is sufficient in its current form.
\end{remark}

\subsection{Asymptotic behavior and properties of spaces of cusp forms}\label{sec:asymptotic}

The asymptotic behavior of an $\ell$-modular lattice $L$ is easiest to describe if all shells of $L$ form spherical $4$-designs.
In that case, when the parameter $\alpha$ is large enough, \eqref{eq:Hessian:4design} is strictly positive, which shows that $L$ is a local minimum for $L \mapsto \mathcal{E}(\alpha, L)$.

Analyzing the asymptotic behavior of $\Ecal(\alpha,L)$ as $\alpha \rightarrow \infty$ is harder if the shells of $L$ only form $2$-designs.
We carry this out for extremal $\ell$-modular lattices with $\ell$ prime and $1 + \ell$ dividing $24$ such that the space of cusp forms $\Scal_{n/2 + 4}(\Gamma_0(\ell),\chi)$, has a basis $C_1,\ldots,C_e$ as in \eqref{eq:row:echelon}, so that Theorem \ref{thm:maintheorem} can be applied. 

\medskip

Let $L$ be an extremal $\ell$-modular lattice with $\ell$ prime and $1 + \ell$ dividing $24$.
Before giving the asymptotic result, we briefly discuss bases and the dimension of $\Scal_{n/2 +4}(\Gamma_0(\ell),\chi)$ (containing theta series with spherical coefficients $\Theta_{L,p}$ coming from harmonic polynomials $p$ of degree $4$) in relation to $\Scal_{n/2}(\Gamma_0(\ell),\chi)$ (containing $\Theta_L$ itself).

\smallskip

We frequently assume that $\Scal_{n/2 +4}(\Gamma_0(\ell),\chi)$ has a basis in row echelon form.
Whether such a basis exists, depending on the weight $n/2 + 4$, the prime $\ell$, and the character $\chi$ is, in general, unanswered.
The existence of a basis in row-echelon form is guaranteed for the full modular group $\SL_2(\Z)$, this is referred to as the (integral) Miller basis and is attributed to Miller \cite{Miller1975}.
However, even for $\Gamma_0(\ell)$, without a character, it can happen that such a basis no longer exists, for a recent treatment we refer to Wang \cite{Wang2023}.
In fact, the discussion after Question 1.5. in \cite{Wang2023} gives that no ``Miller basis'', i.e. basis in reduced row-echelon form, exists for $\Scal_{12k}(\Gamma_0(\ell))$ if $\ell$ is a prime and $12k < \ell -3$. 
The smallest such parameters give the space $\Scal_{12}(\Gamma_0(17))$, which is of dimension $16$ and admits a basis $C_1,\ldots,C_{16}$ where the $q$-expansions of $C_1,\ldots,C_{15}$ start with $q^1,\ldots,q^{15}$, often followed by a multiple of $q^{16}$, but the $q$-expansion of $C_{16}$ starts with $q^{17}$. 

\smallskip

Let now be $m_0$ be such that $2m_0$ is the minimum of $L$, i.e. the squared norm of a shortest nonzero element of $L$.
Then the coefficient $a_m$ of $q^m$ in $\Theta_L$ is zero for $1 \leq m < m_0$ and 
\[
 m_0 - 1 = \dim(\Scal_{n/2}(\Gamma_*(\ell),\chi)) \leq \dim(\Scal_{n/2}(\Gamma_0(\ell),\chi)).
\]
The first equality is contained in \cite[Theorem 2.1]{Scharlau1999}, formulas for $\dim(\Scal_{n/2}(\Gamma_*(\ell),\chi))$ for certain $k$ and $\ell$ can be found in \cite[Theorem 6.]{Quebbemann1995} and \cite[Proposition 1.4]{Scharlau1999}.
We will be interested in the case that
\begin{equation} \label{eq:dim:assumption}
 m_0 - 1 < e = \dim(\Scal_{n/2+4}(\Gamma_0(\ell),\chi),
\end{equation}
which is satisfied if we have 
\begin{equation} \label{eq:dim:strict}
  \begin{aligned}
    &\dim(\Scal_{n/2}(\Gamma_*(\ell),\chi)) < \dim(\Scal_{n/2}(\Gamma_0(\ell),\chi) \\
\text{or} \quad &\dim(\Scal_{n/2}(\Gamma_0(\ell),\chi)) < \dim(\Scal_{n/2+4}(\Gamma_0(\ell),\chi) =e.
  \end{aligned}
\end{equation}

Nonetheless, we can derive a simple criterion to decide when there are strictly less cusp forms for the Fricke group $\Gamma_*(\ell)$ than for $\Gamma_0(\ell)$.
We derive this avoiding dimension formulas, as closed expressions for the dimension of $\Scal_{n/2}(\Gamma_*(\ell),\chi)$ are not readily available in the literature.

\begin{lemma}\label{lem:dimensionality:gap}
  Let $\ell$ be prime such that $1 + \ell$ divides $24$. Let $k_1 = \frac{24}{1 + \ell}$. 
  Then 
  \[
    \Scal_k(\Gamma_*(\ell),\chi) \subsetneq \Scal_k(\Gamma_0(\ell),\chi)
  \]
  if $k \geq k_{\Eis} + k_1$, where $k_{\Eis} = 4$ if $\ell = 2$ and $k_{\Eis} = 3$ otherwise.
\end{lemma}

The constant $k_{\Eis}$ is made so that $k_{\Eis}$ is the smallest weight, such that for all $k'\geq k_{\Eis}$ the space $\Eis_{k'}(\Gamma_0(\ell),\chi))$ is $2$-dimensional.
The significance of this will become apparent in the proof of the lemma.

\begin{proof}
  We recall from \eqref{eq:starspace:polynomialalgebra} that 
  \[
    \Mcal(\Gamma_*(\ell),\chi) = \bigoplus_{k \geq 0} \Mcal_k(\Gamma_*(\ell),\chi) = \C[\Theta,\Delta_{\ell}], 
  \]
  where $\Theta$ is the theta series of an even $\ell$-modular level of lowest possible dimension $2k_0$, i.e. a modular form for $\Gamma_*(\ell)$ of weight $k_0$, and $\Delta_{\ell}(\tau) = (\eta(\tau)\eta(\ell \tau))^{k_1}$ with $k_1 = \frac{24}{1 + \ell}$ is a cusp form for $\Gamma_*(\ell)$ of weight $k_1$. 

  Write $k = k' + k_1$. Since $k' \geq k_{\Eis}$ the space $\Eis_{k'}(\Gamma_0(\ell),\chi))$ is $2$-dimensional.
  Let $E_{k'}^{\chi}$ be an element of this space.
  Then $E_{k'}^{\chi} \Delta \in \Scal_{k}(\Gamma_0(\ell),\chi)$.
  If $E_{k'}^\chi$ is also an element of $\Scal_{k}(\Gamma_*(\ell),\chi)$ we can write
  \[
    E_{k'}^{\chi} \Delta_{\ell} = \sum_{\substack{\lambda,\mu \geq0 \\ \lambda k_0 + \mu k_1 = k}} \Theta^\lambda \Delta_{\ell}^\mu,
  \]
  forcing $\mu \geq 1$ for every summand, since $E_{k'}^{\chi} \Delta_{\ell}$ is a cusp form.
  Cancelling $\Delta_{\ell}$ from both sides of the above identity, we get $E_{k'}^{\chi} \in \Mcal_{k'}(\Gamma_*(\ell),\chi)$. 
  Now the Eisenstein subspace of $\Mcal_{k'}(\Gamma_*(\ell),\chi)$ is $1$-dimensional, while $\Eis_{k'}(\Gamma_0(\ell),\chi))$ is $2$-dimensional, so there exists some $E_{k'}^{\chi} \in \Eis_{k'}(\Gamma_0(\ell),\chi)$ such that $E_{k'}^{\chi} \Delta_{\ell} \notin \Scal_{k}(\Gamma_*(\ell),\chi)$, as desired.
\end{proof}

We briefly note that also the latter inequality in \eqref{eq:dim:strict} can be verified directly, e.g. with the help of a computer, for $\ell$ and $k$ of interest.
If $\chi$ acts trivially on $\Gamma_0(\ell)$ one can also easily reduce this question to evaluating the known dimension formulas for spaces of cusp forms for $\Gamma_0(\ell)$ (e.g. \cite[Theorems 3.5.1 and 3.6.1]{Diamond2005}):
In all such cases the dimension grows by at least $1$ since the dimension of $\Scal_k(\Gamma_0(\ell))$ grows at least with $\lfloor \frac{k}{4}\rfloor$. 

From the preceding discussion we derive the following helpful observation.

\begin{corollary} \label{cor:minimum:vs:dimension}
  Let $\ell$ be a prime such that $1 + \ell$ divides $24$. 
  Let $L$ be an even extremal $\ell$-modular lattice of dimension $n$ with minimum $2m_0$.
  Then $m_0 - 1 < e= \dim(\Scal_{n/2+4}(\Gamma_0(\ell),\chi)$ unless $\ell = 2$ and $L \cong D_4$ or $\ell =3$ and $L \cong A_2$.
\end{corollary}

\begin{proof}
  If $n/2 + 4 \geq k_{\Eis} + k_1$ this immediately follows from Lemma \ref{lem:dimensionality:gap}.
  There are only finitely many cases left, that are not covered by this.
  In these cases we check directly whether the latter inequality in \eqref{eq:dim:strict} is satisfied.
  Direct computations then give that only two (trivial) exceptions remain, these are $\ell = 2$ and $n = 4$, which is $D_4$, and $\ell = 3$ and $n = 2$, which is $A_2$.
  In both cases $\dim(\Scal_{n/2+4}(\Gamma_0(\ell),\chi) = 0$
\end{proof}

\medskip

We can now turn back to observe the following asymptotic behavior. 

\begin{theorem}
  \label{thm:large-alpha}
  Let $L$ be an extremal even $\ell$-modular lattice of dimension $n$ and let $m_0$ be such that $2m_0$ is the minimum of $L$.
  Assume that $L$ satisfies the assumptions of Theorem \ref{thm:maintheorem}.
  Then the lattice $L$ is a local minimum for all large enough $\alpha$ if and only if all eigenvalues of $Q_{2m_0}$ are strictly positive. 
  If one eigenvalue of $Q_{2m_0}$ vanishes, then $L$ is a saddle point for all large enough $\alpha$.
\end{theorem}

\begin{proof}
First we note that this is obviously true for $L$ isometric to $A_2$ or $D_4$, as these lattices are locally universally optimal among lattices by results of Montgomery \cite{Montgomery1988} and Sarnak and Str\"ombergsson \cite{Sarnak2006a}.
For the remaining cases we use our analysis of the eigenvalues of the Hessian of energy as in Theorem \ref{thm:maintheorem}.

\smallskip
The growth of $c_{i,m}$ and $a_m$ is polynomial in $m$ (as we will see in Section \ref{sec:bounds}).
Since $L$ is extremal, $Q_{2m} = 0$ for $1 \leq m < m_0$, but $Q_{2m_0}$ is a non-trivial quadratic form.
We dissect \eqref{eq:hessian:eigenvalues}. 
Firstly we have 
\[
  \frac{1}{n(n+2)}\left(\sum_{i = 1}^{e} \left(c_{i,m}\frac{\alpha^2}{2}\left(\lambda_{i,k}n(n+2)- 8 i^2 a_i \right)\right)\right) e^{-2\alpha m} = 0
\]
if $m < m_0$, since then $Q_{2m} = 0$, and so are all associated eigenvalues, and $a_m = 0$.
Secondly, 
\[
  \frac{1}{n(n+2)}(a_m2\alpha m(2\alpha m - (n/2 + 1))) e^{-2\alpha m} = 0
\]
since $a_m = 0$.

So the first contribution to the $k$-th eigenvalue of $\nabla^2\Ecal(\alpha,L)$ comes from the index $m_0$.
For this we see that the sum 
\[
  \sum_{i = 1}^{e} \left(c_{i,m_0}\frac{\alpha^2}{2}\left(\lambda_{i,k}n(n+2)- 8 i^2 a_i \right)\right)
\]
collapses to a single summand, that for $i = m_0$.
This is because the basis $C_1,\ldots,C_e$ is in reduced row echelon form and
\[
  e = \dim(\Scal_{n/2 +4}(\Gamma_0(\ell),\chi)) > \dim(\Scal_{n/2}(\Gamma_0(\ell),\chi)) = m_0 -1,
\]
so only the cusp form $C_{m_0}$, which has a Fourier expansion of the form $q^{m_0} + O(q^{e+1})$, contributes.

With the above and the estimate provided in Lemma~\ref{lem:approx:tail}, we see that the dominating summand in \eqref{eq:hessian:eigenvalues} is of the form
\[
\frac{1}{n(n+2)}  \left(\frac{\alpha^2}{2} (\lambda_{m_0,k} n(n+2)) - 8 m_0^2 a_{m_0} \alpha  (n/2 + 1) \right) e^{-2\alpha m_0}.
\]

In particular, for large $\alpha$, this summand is strictly positive if $\lambda_{m_0,k}$ is strictly positive and the first summand is strictly negative if $\lambda_{m_0,k}$ vanishes.
As the quadratic form $Q_{2m_0}$ is a non-trivial sum of squares, the eigenvalues cannot be strictly negative, and some eigenvalue is always strictly positive. 
\end{proof}

\section{Explicit bounds on the coefficients of modular forms} \label{sec:bounds}
To make the results on the eigenvalues of the Hessian $\nabla^2 \Ecal(\alpha,L)$ usable for explicit computations we need explicit bounds for certain modular forms. 

For this we use the decomposition
\begin{equation} \label{eq:modularform:decomposition}
  \Mcal_k(\Gamma_0(\ell),\chi) = \Eis_k(\Gamma_0(\ell),\chi) \oplus \Scal_k(\Gamma_0(\ell),\chi),
\end{equation}
for the group $\Gamma_0(\ell)$ and a character $\chi$.
Here $\Eis_k(\Gamma_0(\ell),\chi)$ is the \emph{Eisenstein subspace}, which is the orthogonal complement of the subspace of cusp forms $\Scal_k(\Gamma_0(\ell),\chi)$ with respect to the Petersson inner product (c.f. \cite[Section 5.11]{Diamond2005}).

This decomposition allows us to decompose the theta series $\Theta_L$ of an extremal $\ell$-modular lattice into an Eisenstein part and a cusp form part and then use separate bounds on coefficients of Eisenstein series and cusp forms.

\subsection{Coefficients of Eisenstein series for $\Gamma_0(\ell)$} \label{sec:modularform:coeff}

  The previously cited definition of the Eisenstein subspace $\Eis_k(\Gamma_0(\ell),\chi)$ as the orthogonal complement of $\Scal_k(\Gamma_0(\ell),\chi)$ is non-constructive.
  Here we make use of a more explicit description in terms of a basis of $\Eis_k(\Gamma_0(\ell),\chi)$, where we can explicitly describe the Fourier coefficients of its members.

  We claim that 
  \begin{equation} \label{eq:eisenstein:basis:gen}
    \Eis_k(\Gamma_0(\ell)) = \operatorname{span}\lset E_k(\tau), E_k(\ell\tau) \rset,
  \end{equation}
  where $E_k$ is the normalized Eisenstein series for $\SL_2(\Z)$ of weight $k$.
  One particular easy way to justify this claim is to use that the dimension formula for $\Eis_k(\Gamma_0(\ell))$ to see that this is a $2$-dimensional space for any prime $\ell$ and then check that both (obviously linearly independent) elements above are modular forms for $\Gamma_0(\ell)$.

  From \eqref{eq:eisenstein:basis:gen} we see that the coefficients of both basis elements can be trivially bounded by the coefficients of $E_k = \sum_{n \geq 0} b_n q^n$. 
  For the latter we use the explicit form $b_n = -\frac{2k}{B_k}\sigma_{k-1}(n)$.
  Using the upper bound $\sigma_{k-1}(n) \leq \zeta(k-1)n^{k-1}$, where $\zeta$ is the Riemann zeta function, we get 
  \begin{equation} \label{eq:eisenstein:coeff:bound}
      \vert b_n \vert = \left\vert -\frac{2k}{B_k}\sigma_{k-1}(n)\right\vert \leq \left\vert -\frac{2k}{B_k}\zeta(k-1)n^{k-1}\right\vert.
  \end{equation}

  For applications to lattices it is more useful to use a modified basis for $\Eis(\Gamma_0(\ell))$.
  We use the basis consisting of
  \begin{equation} \label{eq:eisenstein:basis}
    E^{(1)}_k(\tau) = E_k(\ell\tau) \quad \text{and} \quad E^{(2)}_k(\tau) = \frac{1}{b_1}\left(E_k(\ell\tau) - E_k(\tau) \right),
  \end{equation} 
  where $b_1$ is the coefficient of $q$ in $E_k$.
  The coefficients of $E_k^{(1)}$ and $E^{(2)}_k$ can be bounded (very wastefully) in terms of \eqref{eq:eisenstein:coeff:bound}.
  For this write 
  \begin{equation*}
      E^{(1)}_k(\tau) = \sum_{m = 0}^{\infty} b^{(1)}_m q^m, \quad \text{and} \quad
      E^{(2)}_k(\tau) = \sum_{m = 0}^{\infty} b^{(2)}_m q^m,
  \end{equation*}
 where explicitly $b^{(1)}_{3m} = b_{m}$ and $ b^{(2)}_m = \frac{1}{b_1}\left( b_{\ell m} - b_m \right)$.  From this we get 
  \begin{equation} \label{eq:coeff:bound:Eisenstein}
    |b^{(1)}_m|, |b^{(2)}_m| \leq  | b_m | \leq \left\vert-\frac{2k}{B_k}\zeta(k-1)n^{k-1}\right\vert.
  \end{equation}

  We include a second justification of \eqref{eq:eisenstein:basis:gen} which uses a bit more of a  modular forms background.
  The advantage of this is that it works in the presence of a non-trivial character $\chi$ and that it would allow obtaining more rigorous bounds on the coefficients of Eisenstein series, if that was needed.
  For this we construct a basis of $\Eis_k(\Gamma_0(\ell),\chi)$.
  By \cite[Theorems 4.5.1 and 4.5.2]{Diamond2005} such a basis is given by 
  \[
    \lset E_k^{\psi,\varphi,t} : (\psi,\varphi,t) \in A_{N,k}, \psi\varphi = \chi \rset,
  \]
  where $A_{N,k}$ is the set of triples $(\psi,\varphi,t)$ such that $\psi$ and $\varphi$ are primitive\footnote{Recall that a Dirichlet character $\chi$ is \emph{primitive} modulo $N$ if its conductor is $N$, i.e. if $\chi$ is not the lift of some Dirichlet character for a divisor of $N$.} Dirichlet characters modulo $u$ and $v$, such that $(\psi \varphi)(-1) = (-1)^k$ and $t \in \Z_{>0}$ such that $tuv \mid \ell$.

  Here, $E_k^{\psi,\varphi,t}(\tau) = E_k^{\psi,\varphi}(t\tau)$, and the Fourier expansion of $E_k^{\psi,\varphi}$ is 
  \begin{equation} \label{eq:fourier:expansion:eisenstein}
    E_k^{\psi,\varphi} = \delta(\psi)L(1-k,\varphi) + 2 \sum_{n=1}^\infty \sigma_{k-1}^{\psi,\varphi}(n)q^n,
  \end{equation}
  where $\delta(\psi) = 1$ if $\psi$ is the trivial Dirichlet character modulo $1$ and $0$ otherwise. 
  Furthermore, we use the generalized divisor power sum 
  \begin{equation} \label{eq:generalized:divisor-power-sum}
    \sigma_{k-1}^{\psi,\varphi}(n) = \sum_{\substack{m \mid n \\ m>0}} \psi(n/m)\varphi(m) m^{k-1}.
  \end{equation}

  If $\chi$ is the trivial character modulo $\ell$, there is only one way to write $\chi$ as the product $\chi = \psi\varphi$ (the product understood modulo $\ell$) of characters modulo $u$ and $v$, namely if both $\psi$ and $\varphi$ are the trivial character modulo $1$\footnote{While the trivial character modulo $1$ is primitive, no trivial character modulo $N>1$ is primitive, as the conductor is always $1$.}. 
  So if $\ell$ is a prime this tells us that $\dim(\Eis_k(\Gamma_0(\ell))) = 2$ and a basis is indexed by the triples $(1_1,1_1,1)$ and $(1_1,1_1,\ell)$.
  Finally, we observe
  \[
    E_k^{1_1,1_1}(\tau) = E_k(\tau) \quad \text{and} \quad E_k^{1_1,1_1,\ell}(\tau) = E_k(\ell \tau).
  \]
  
\subsection{Coefficients of cusp forms for $\Gamma_0(2)$ and $\Gamma_0(3)$} \label{sec:cuspform:coeff}

We shortly describe an approach to bound the coefficients of cusp forms based on a decompostion into Hecke eigenforms and an application of Deligne's theorem \cite{Deligne1974}.
This approach was also used by de Courcy-Ireland, Dostert, and Viazovska \cite{deCourcyIreland:sds:2024}.

For this, suppose that $C' \in \Scal_{k}(\Gamma_0(\ell),\chi)$ is a cusp form and that we have a decomposition 
\begin{equation} \label{eq:hecke}
    C' = \lambda_1 C_1 + \cdots + \lambda_k C_k
\end{equation}
into distinct Hecke eigenforms $C_1,\ldots,C_k$.
For background on Hecke eigenforms, and a discussion on the existence of a basis of simultaneous eigenforms, we refer to Chapter 5 in \cite{Diamond2005}.
The challenge now is to assert that a given decomposition \eqref{eq:hecke} is in fact a decomposition into Hecke eigenforms, we will assert that on a case by case basis where needed.

Once a decomposition of the form \eqref{eq:hecke} is available, we can use the following bound on each $C_i$.
Let $C(\tau) = \sum_{m=1}^\infty c_m q^m \in \Scal_{k}(\Gamma_0(\ell),\chi)$ be a Hecke eigenform, such that the first non-zero $c_m$ is equal to $1$. 
Then by Deligne's theorem
\begin{equation}\label{eq:deligne}
    |c_m| \leq d(m)m^{(k-1)/2} \leq 2 m^{k/2},
\end{equation}
where we use the (rough) bound $d(m) \leq 2 \sqrt{m}$.
Note that this latter bound can be strengthened (for large $m$), but we do not require the additional precision at this point.

As a closing remark we mention that for cases where a decomposition of a specific cusp form $C$ into Hecke eigenforms cannot be computed, one can still use a general-purpose estimate for cusp forms of $\Gamma_0(2)$ and $\Gamma_0(3)$, if no character is present. These bounds are provided by Jenkins and Pratt \cite{Jenkins2015} for $\Gamma_0(2)$, and Choi and Im \cite{Choi2018} for $\Gamma_0(3)$.
Similarly, such global bounds exist for $\SL_2(\Z)$, provided by Jenkins and Rouse \cite{Jenkins2011a}, which were used in \cite{Heimendahl2023} to bound coefficients of cusp forms in the context of energy minimization for even unimodular lattices.
In a first version of this article, this was the approach we used, until the above approach was suggested to us, independently, by Matthew de Courcy-Ireland and the anonymous referee.

\section{Investigation of modular lattices} \label{sec:results}

We now apply the strategy outlined in Section \ref{sec:Strategy} to extremal (strongly) modular lattices of level $2$ and $3$. 
One main reason to investigate such lattices in more detail is that these include two very well-known lattices:
For one the $3$-modular Coxeter-Todd lattice in dimension $12$ and the $2$-modular Barnes-Wall lattice in dimension $16$, both of which are the densest known sphere packings in their respective dimensions.

We provide explicit examples of the (local) behavior of the energy function $\Ecal(\alpha,L)$ for certain energy levels $\alpha$ and extremal $2$ and $3$-modular lattices. 
For this analysis we will need to compute the eigenvalues of the Hessian $\nabla^2 \Ecal(\alpha,L)$ and then evaluate certain resulting infinite series explicitly. 
To this end, we will use the explicit bounds on the coefficients of the theta series of $L$, which we obtained in Section \ref{sec:bounds}.

\subsection{On extremal $2$ and $3$-modular lattices in small dimensions} \label{sec:2-3-modular}

We now collect some information on extremal $2$ and $3$-modular lattices in small dimensions. 
To be precise, extremal $2$-modular lattices have been completely classified in dimensions $\leq 20$ and extremal $3$-modular lattices have been completely classified  in dimensions $\leq 18$.
Data on these lattices is readily available through the ``catalogue of lattices'' \cite{Nebe}. 

To compute the eigenvalues of the Hessian \eqref{eq:Hessian} for such a lattice $L$, we need to know the design strength of the shells $L$ (c.f.\ Section \ref{sec:spherical:shells}). 
If the design strength is not at least $4$, we need additional information to be able to apply Theorem \ref{thm:maintheorem}.
For this, recall that for such $L$ the theta series with spherical coefficients $\Theta_{L,p}$ for a spherical polynomial of degree $d$ is a modular form in the space $\Scal_{n/2 + d}(\Gamma_0(\ell),\chi)$.
To apply Theorem \ref{thm:maintheorem} we specifically need information on the case $d = 4$, and so we collect whether $\chi$ is trivial (if and only if $n = 0 \mod 4$), and therefore $\Scal_{n/2+4}(\Gamma_0(\ell),\chi) = \Scal_{n/2+4}(\Gamma_0(\ell))$, or not.
In addition, we collect the dimension of that space if $\chi$ is indeed trivial.

The relevant data is presented in Tables \ref{tab:modular:lattices:2} and \ref{tab:modular:lattices:3} for $2$ and $3$-modular lattices separately.
In both cases, we record for each dimension the \emph{class number}, i.e. number of non-isometric extremal lattices, the associated space of cusp forms, its dimension, whether $\chi$ is trivial, and the minimum guaranteed design strength on the shells of such a lattice, see Bachoc and Venkov \cite{Bachoc2001}.

\begin{table}\renewcommand{\arraystretch}{1.1}
  \begin{tabular}{c c c c c }\hline
    dimension & class number & $\chi$ trivial? & $\dim(\Scal_{n/2 + 4}(\Gamma_0(2),\chi))$ & design strength \\ \hline \rule{0pt}{4mm}
    $4$ & $1$ & yes & $0$ & $\geq 5$ \\
    $8$ & $1$ & yes & $1$ &  $\geq 3$ \\
    $12$ & $3$ & yes & $1$ & $\geq 1$ \\
    $16$ & $1$ & yes & $2$ & $\geq 7$ \\
    $20$ & $3$ & yes & $2$ & $\geq 5$\\ \hline
  \end{tabular}
  \medskip
  \caption{Data on $2$-modular extremal lattices in small dimensions.}
  \label{tab:modular:lattices:2}
\end{table}

\begin{table}\renewcommand{\arraystretch}{1.1}
  \begin{tabular}{c c c c c }\hline
    dimension & class number & $\chi$ trivial? & $\dim(\Scal_{n/2 + 4}(\Gamma_0(3),\chi))$ & design strength \\ \hline \rule{0pt}{4mm}%
    $2$ & $1$ & no & $0$ & $\geq 5$ \\
    $4$ & $1$ & yes & $1$ &  $\geq 3$ \\
    $6$ & $1$ & no & $1$ & $\geq 3$ \\
    $8$ & $2$ & yes & $1$ & $\geq 1$ \\
    $10$ & $3$ & no & $2$ & $\geq 1$ \\
    $12$ & $1$ & yes & $2$ & $\geq 5$ \\
    $14$ & $1$ & no & $2$ &  $\geq 5$ \\
    $16$ & $6$ & yes & $3$ & $\geq 3$ \\
    $18$ & $37$ & no & $3$ & $\geq 3$  \\\hline
  \end{tabular}
  \medskip
  \caption{Data on $3$-modular extremal lattices in small dimensions.}
  \label{tab:modular:lattices:3}
\end{table}

\subsection{The Coxeter-Todd lattice}

The Coxeter-Todd lattice, often denoted $K_{12}$ first appeared in \cite{Coxeter1953}.
It is an extremal $3$-modular lattice in dimension $12$, and it is the densest known sphere packing in this dimension at the time of writing this article.
For more information on this lattice we refer to \cite[Ch. 4.9]{Conway1988a}.

The theta series of $K_{12}$ is  
\begin{equation*}
    \Theta_{K_{12}}(\tau) = 1 + 756q^2 + 4032q^3 + 20412q^4 + 60480q^5 + 139860q^6 + O(q^7)
\end{equation*}
and is a modular form for $\Gamma_0(3)$ of weight $6$. 
For any $m \geq 1$, the Coxeter-Todd lattice contains an element of norm $2m$ and the corresponding shell $K_{12}(2m)$ is a spherical $5$-design.
Therefore, by Lemma \ref{lem:Hessian:4design}, all eigenvalues of the Hessian $\nabla^2 \Ecal(\alpha,K_{12})$ coincide.

\begin{proposition} \label{prop:CT}
The Coxeter-Todd lattice $K_{12}$ is locally universally optimal among lattices.
\end{proposition}

\begin{proof}

The eigenvalues of the Hessian $\nabla^2 \Ecal(\alpha,K_{12})$ given by
\begin{equation} \label{eq:CT:EV}
  \frac{1}{168} \sum_{m=0}^{\infty}  a_m  2\alpha m (2\alpha m - 7) e^{-2\alpha m} ,
\end{equation}
where $a_m = \vert K_{12}(2m)\vert$ is the $m$-th coefficient of the theta series of $K_{12}$.
Note that the summands for $m=0$ and $m=1$ vanish, the latter because $a_1 = \vert K_{12}(2) \vert = 0$, as the Coxeter-Todd lattice does not contain roots.

Now, if $\alpha \geq \pi$, each summand of \eqref{eq:CT:EV} is positive, the only factor that could become negative is $(2\alpha m - 7)$, but this only happens if $m<2$ by
\begin{equation*}
  2\alpha m - 7  < 0 \quad \Leftrightarrow \quad m < \frac{7}{2 \alpha} \leq \frac{7}{2 \pi} < 1.2.
\end{equation*}

If $\alpha < \pi$ we can use Poisson summation, to reduce to the case above.

\end{proof}

\subsection{The Barnes-Wall lattice}
The $16$-dimensional Barnes-Wall lattice, often denoted $\operatorname{BW}_{16}$ or $\Lambda_{16}$ first appeared in \cite{Barnes1959}.
It is an extremal $2$-modular lattice in dimension $16$ and the densest known sphere packing in this dimension at the time of writing this article.
For more information on this lattice we refer to \cite[Ch. 4.10]{Conway1988a}.

\smallskip

The theta series of $\operatorname{BW}_{16}$ is
\begin{align*}
  \Theta_{\operatorname{BW}_{16}} =1 + 4320q^2 + 61440q^3 + 522720q^4 + 2211840q^5 + 8960640q^6 + O(q^7)
\end{align*}
and is a modular form for $\Gamma_0(2)$ of weight $8$. 
For any $m \geq 1$ the Barnes-Wall lattice contains an element of norm $2m$ and the corresponding shell $\operatorname{BW}_{16}(2m)$ is a spherical $7$-design.
Therefore, by Lemma \ref{lem:Hessian:4design}, all eigenvalues of the Hessian $\nabla^2 \Ecal(\alpha,\operatorname{BW}_{16})$ coincide.

\begin{proposition} \label{prop:BW}
  The Barnes-Wall lattice $\operatorname{BW}_{16}$ is locally universally optimal among lattices.
\end{proposition}

\begin{proof}
  The proof is essentially the same as for the Coxeter-Todd lattice in Proposition \ref{prop:CT}.

  The eigenvalues of the Hessian $\nabla^2 \Ecal(\alpha,\operatorname{BW}_{16})$ given by
  \begin{equation} \label{eq:BW:EV}
    \frac{1}{168} \sum_{m=0}^{\infty}  a_m  2\alpha m (2\alpha m - 9) e^{-2 \alpha m} ,
  \end{equation}
  where $a_m = \vert \operatorname{BW}_{16}(2m)\vert$ is the $m$-th coefficient of the theta series of $\operatorname{BW}_{16}$.
  Note that the summands for $m=0$ and $m=1$ vanish, the latter because $a_1 = \vert \operatorname{BW}_{16}(2) \vert = 0$, as the Barnes-Wall lattice does not contain roots.

  Now, if $\alpha \geq \pi$, each summand of \eqref{eq:BW:EV} is positive, the only factor that could become negative is $(2\alpha m - 9)$, but this only happens if $m<2$ by
  \begin{equation*}
    2\alpha m - 9  < 0 \quad \Leftrightarrow \quad m < \frac{9}{2 \alpha} \leq \frac{9}{2 \pi} < 1.5.
  \end{equation*}

  If $\alpha < \pi$ we can use Poisson summation, to reduce to the case above.
\end{proof}

\subsection{Local universal optimality of extremal $\ell$-modular lattices} \label{sec:loc:univ}

We briefly note that the results in the preceding sections regarding the local universal optimality among lattices of the Coxeter-Todd lattice and the Barnes-Wall lattice readily extend as follows.

\begin{theorem} \label{thm:ellmodular:loc}
  Let $L$ be an extremal $\ell$-modular lattice for which all shells are $4$-designs. Then $L$ is locally universally optimal among lattices if 
  $\ell \geq  3$, or if $\ell = 2$ and $\dim(L) \in \lset 4, 16, 20, 32, 48 \rset$.
\end{theorem}

To apply this theorem we need to have (lower bounds) on the minimal design strength on the shells of extremal $\ell$-modular lattices.
There are two sources for this, as already discussed in Section \ref{sec:spherical:shells}.
First, by \cite[Corollary 4.1]{Bachoc2001} for all extremal $\ell$-modular lattices with $\ell = 2$ and $n \equiv 0,4 \mod 16$, or $\ell = 3$ and $n \equiv 0,2 \mod 12$, all shells form  spherical designs of strength at least $4$.
Second, by computing the harmonic Molien series and using the criterion in Theorem 3.12 in \cite{Goethals1981}, we can investigate all remaining cases in low dimensions case by case.
Combining the two approaches leads to a complete classification of all extremal $\ell$-modular lattices in low dimensions, it is collected in Table \ref{tab:loc:univ}.

\begin{proof}
  For this we again note that the unique eigenvalue of the Hessian is given by 
  \begin{equation*}
      \frac{1}{n(n+2)} \sum_{m = 1} \vert L(2m) \vert \alpha 2m (\alpha 2m - (n/2 + 1)) e^{-\alpha 2m}, 
  \end{equation*}
  as in \eqref{eq:Hessian:4design}.
  can only be negative if $\alpha 2m - (n/2+1)  < 0$.
  The minimum of $L$, that is the first $m>0$ for which $L(2m)$ is non-empty, is given by $1 + \lfloor \frac{n/2}{24/(\ell + 1)} \rfloor$ (see \cite{Quebbemann1995} or \cite{Scharlau1999}).
  Therefore, the minimum increases every time the dimension increases by one of the jump dimensions $j(\ell)$ below.
  \begin{center}\renewcommand{\arraystretch}{1.1}
    \begin{tabular}{l c c c c c c }\hline
      $\ell$ & $2$ & $3$ & $5$ & $7$ & $11$ & $23$ \\ \hline \rule{0pt}{4mm}%
      $j(\ell)$ & 16 & 12 & 8 & 6 & 4 & 2 \\\hline
    \end{tabular}
  \end{center}
  From this we get that for $\alpha \geq \pi$, the inequality $\alpha 2m - (n/2+1)  > 0$ is satisfied if and only if
  \begin{equation*}
    m > \frac{n/2 +1}{2 \pi} \quad \Longleftrightarrow \quad n < 4 \pi m - 2.
  \end{equation*}
  Now we can check for every $m$, which dimension, for which the minimum of $L$ is precisely $m$, satisfies this inequality.

  The largest dimension for which $L$ has minimum $m$ is at most $j(\ell) \cdot m - 2$.
  Plugging this in we obtain 
   \begin{equation*}
    j(\ell) \cdot m - 2 < 4 \pi m - 2 \quad \Longleftrightarrow j(\ell) < 4 \pi ,
  \end{equation*}
  where the latter is always true if $\ell \geq 3$, since then $j(\ell) \leq 12 < 4\pi$.

  If $\ell = 2$, this is the case precisely for $m = 1$ and $n \in \lset 4,8 \rset$, for $m = 2$ and $n \in \lset 16, 20 \rset$, for $m = 3$ and $n = 32$, and for $m = 4$ and $n = 48$.    

  As before, the proof can now be finished by utilizing Poisson summation, to reduce the case $\alpha < \pi$ to the case above.
\end{proof}

{\begin{table}\renewcommand{\arraystretch}{1.1}
  \begin{tabular}{c c c c c}\hline
    $\ell$ & dimension & class number & lattice & design strength \\ \hline \rule{0pt}{4mm}%
    $2$ & $4$ & $1$ & $D_4$ & $\geq 5$ \\
    $2$ & $16$ & $1$ & Barnes-Wall & $\geq 5$ \\
    $2$ & $20$ & $3$ & all & $\geq 5$ \\
    $3$ & $2$ & $1$ & $A_2$ & $\geq 5$  \\
    $3$ & $12$ & $1$ & Coxeter-Todd & $\geq 7$  \\
    $3$ & $14$ & $1$ & unique & $\geq 5$  \\
    $5$ & $16$ & $1$ & unique & $\geq 5$  \\
    \hline
  \end{tabular}
  \medskip
  \caption{Locally universally optimal extremal $\ell$-modular  lattices in dimensions at most $20$.}
  \label{tab:loc:univ}
\end{table}
}

\subsection{Extremal modular lattices without $4$-designs} \label{sec:2design}

We now illustrate how to determine the local behavior of extremal $2$ and $3$-modular lattices, when the shells are not necessarily $4$-designs, but are at least $2$ designs.
Concretely, we investigate an extremal $3$-modular lattice $L_{16}$ of dimension $16$, given explicitly as the third entry in ``\texttt{3\_dim16.dat}'' in \cite{Nebe}. 
For this lattice, all shells form at least $3$-designs, and, as we will check below, Theorem \ref{thm:maintheorem} is applicable.

Numerical computations indicate that this lattice, or any other $3$-modular lattice in dimension $16$ (and $18$), is a saddle point.
These numerical considerations can be turned into rigorous results.
We illustrate the method, which relies on the formula for the eigenvalues developed in Theorem \ref{thm:maintheorem}, and provide a rigorous proof that this lattice is a saddle point at the value $\alpha = 1$.

We stress that the method works more generally.
First, the choice $\alpha = 1$ only simplifies the write-up a bit.
Second, we choose this specific lattice, among the extremal $3$-modular lattices in dimension $16$ and $18$, because in this case the Hessian has only $3$ distinct eigenvalues, so there are fewer estimates to provide for a full investigation.

\smallskip

The theta series of $L_{16}$ is 
\begin{equation} \label{eq:theta:L16}
    \Theta_{L_{16}}(\tau) = 1 + 720q^2 + 13440q^3 + 97200q^4 + O(q^5),
\end{equation}
and it is a modular form of weight $8$ for $\Gamma_0(3)$. 
$\Mcal_8(\Gamma_0(3))$ decomposes into a direct sum of the $2$-dimensional subspace of Eisenstein series $\Eis_8(\Gamma_0(3))$ and the one dimensional cuspidal subspace $\Scal_8(\Gamma_0(3))$ (c.f. \eqref{eq:modularform:decomposition}).

\begin{proposition} \label{prop:L16}
  Let $L_{16}$ be the $16$-dimensional $3$-modular lattice which appears as the third entry in ``\texttt{3\_dim16.dat}'' in \cite{Nebe}.
  For every $\alpha > 0$, the Hessian $\nabla^2 \Ecal(\alpha,L_{16})$ has at most three distinct eigenvalues.

  For $\alpha = 1$, the Hessian $\nabla^2 \Ecal(1,L_{16})$, has one of positive and two negative eigenvalues.
  In particular: $L_{16}$ is a saddle point for $\Ecal(1,L)$ on the manifold of lattices with point density $3^4$.
  
\end{proposition}

\begin{proof}
We firstly bound the coefficients of $\Theta_{L_{16}}$ and then use Theorem \ref{thm:maintheorem} to estimate the eigenvalues of the Hessian $\nabla^2 \Ecal(1,L_{16})$ at $L_{16}$, rigorously so for the specific energy parameter $\alpha = 1$.
For this it turns out that we need the eigenvalues of the auxiliary forms $Q_{4}$ and $Q_{6}$ associated to $L_{16}$, requiring further estimating coefficients of cups forms in $\Scal_{12}(\Gamma_0(3))$.

\subsubsection{Bounding the coefficients of $\Theta_{L_{16}}$}

We use the basis $E^{(1)}_8, E^{(2)}_8$ of $\Eis_8(\Gamma_0(3))$ given in \eqref{eq:eisenstein:basis}, while a basis of $\Scal_8(\Gamma_0(3))$ is given by 
\begin{align*}
  C_8(\tau) = \sum_{m = 0}^{\infty} c_m q^m= q + 6q^2 - 27q^3 - 92q^4 + 390 q^5 - 162 q^6 + O(q^7).
\end{align*}

Then
\begin{align*}
  \Theta_{L_{16}}(\tau) = E^{(1)}_8(\tau) + \frac{720}{123} E^{(2)}_8(\tau) - \frac{720}{123} C_8(\tau).
\end{align*}
To find an upper bound for the coefficients of $\Theta_{L_{16}}$, we will find upper bounds for the coefficients of $E^{(1)}_8$, $E^{(2)}_8$ and $C_8$.    

For the coefficients of $E^{(1)}_8,E^{(2)}_8$ we use the estimate in \eqref{eq:coeff:bound:Eisenstein}. From this we get
\begin{align*}
  |b^{(1)}|,|b^{(2)}_m| \leq \vert b_m \vert \leq 485 m^7
\end{align*}

Again, $\Scal_8(\Gamma_0(3))$ is $1$-dimensional, so $C_8(\tau) = \sum_{m=1}^\infty c_m q^m$ is a Hecke eigenform, and \eqref{eq:deligne} gives
\begin{align*}
    \vert c_m \vert &\leq 2 m^4. 
\end{align*}
Putting everything together, we find the following upper bound for the coefficients of $\Theta_{L_{16}}$:
\begin{equation} \label{eq:theta:L16}
    \vert a_m \vert \leq 3325 \cdot m^7 + 12 \cdot m^4. 
\end{equation}

\subsubsection{The auxiliary forms $Q_4$ and $Q_6$}

To apply Theorem \ref{thm:maintheorem} we observe that the dimension of $\Scal_{12}(\Gamma_0(\ell))$ is $3$.
So we need the auxiliary forms (c.f. \eqref{eq:quadraticforms})
\[
  Q_{2m} = \sum_{x \in L_{16}(2m)} H[x]^2
\]
specifically for $m=1,2,3$.

We consider the explicit basis of $\Scal_{12}(\Gamma_0(3))$ consisting of
\begin{align*}
    C^{(1)}_{12}(\tau) = \sum_{m > 0} c^{(1)}_m q^m &= q - 176q^4 + 2430q^5 - 5832q^6 + O(q^7),\\
    C^{(2)}_{12}(\tau) = \sum_{m > 0} c^{(2)}_m q^m &= q^2 + 54q^4 - 100q^5 - 243q^6 + O(q^7),\\
    C^{(3)}_{12}(\tau) = \sum_{m > 0} c^{(3)}_m q^m &= q^3 - 24q^6 + O(q^7).
\end{align*}
Note that this basis is in reduced row echelon form in agreement with \eqref{eq:row:echelon}.
However, this is not a basis of Hecke eigenforms. 
One such basis can be obtained using the \texttt{mfbasis} command in PARI/GP \cite{PARI} we compute the following basis of Hecke eigenforms
\begin{equation} \label{eq:Hecke:basis}
    \begin{aligned}
        H^{(1)}_{12}(\tau) &= \Delta(\tau) =  \sum_{m > 0} h^{(1)}_m q^m = q - 24q^2 + 252q^3 - 1472q^4 +  O(q^5),\\
        H^{(2)}_{12}(\tau) &= \Delta(3\tau) = \sum_{m > 0} h^{(2)}_m q^m = q^3 + 24q^6 + O(q^7),\\
        H^{(3)}_{12}(\tau) &= \sum_{m > 0} h^{(3)}_m q^m = q +78q^2 -243q^3 + 4036q^4 + O(q^5).
    \end{aligned}
\end{equation}
These are in fact Hecke eigenforms. 
Firstly $H^{(1)}_{12}$ and $H^{(2)}_{12}$ are expressed in terms of Ramanujan's discriminant function $\Delta$, while the third one is collected in the \emph{L-functions and modular forms database} \cite{LMFDB}, as entry \href{https://www.lmfdb.org/ModularForm/GL2/Q/holomorphic/3/12/a/a/}{Newform orbit 3.12.a.a}.

To estimate the eigenvalues of the Hessian, we need to find estimates on the coefficients of the cusp forms $C^{(1)}_{12}$, $C^{(2)}_{12}$, and $C^{(3)}_{12}$.

We express $\Theta_{L_{16},p_H}(\tau) = b_1 C^{(1)}_{12} + b_2 C^{(2)}_{12} + b_3 C^{(3)}_{12}$ in terms of this basis and obtain the coefficients
\begin{align*}
    b_i = \sum_{x\in L(2i)} H[x]^2 - \frac{2(2i)^2}{(2+n)n} \vert L(2i) \vert \Tr(H^2)
\end{align*}
as in \eqref{eq:biconstants}.
Note, however, that $b_1 = 0$ as $L_{16}$ has no elements of squared norm $2$, its minimum is $4$ as follows from its theta series \eqref{eq:theta:L16}.
Therefore, we only need to investigate $C^{(2)}_{12}$ and $C^{(3)}_{12}$ and the associated auxiliary forms $Q_4$ and $Q_6$.

We explicitly computed the matrices of the associated bilinear forms in \eqref{eq:bil-form-b_Q} for $Q_4$ and $Q_6$ and checked (numerically) that they commute, so we can diagonalize both forms simultaneously (see Section \ref{sec:simultaneous:diagonalizable} for a brief discussion of obstacles for an abstract proof of this property). 
These computations were done using a computer, and we obtained the eigenvalues (with multiplicities)
\begin{align*}
  \lambda_{2,1} = 0, \quad \lambda_{2,2} = 72, \quad \lambda_{2,3} = 144
\end{align*}
for $Q_4$ and 
\begin{align*}
    \lambda_{3,1} = 4320, \quad \lambda_{3,2} = 3456, \quad \lambda_{3,3} = 2592,
\end{align*}
for $Q_6$ and their common eigenspaces $E_1$, $E_2$, and $E_3$ respectively.

For the coefficients of $C^{(2)}_{12}$ and $C^{(3)}_{12}$ we express these cusp forms in the basis \eqref{eq:Hecke:basis} of Hecke eigenforms:
\begin{equation*}
    C^{(2)}_{12} = \frac{1}{102}(- H^{(1)}_{12} + 165 H^{(2)}_{12} + H^{(3)}_{12})
\end{equation*}
and
\begin{equation*}
    C^{(3)}_{12} = H^{(2)}_{12}.
\end{equation*}
Using \eqref{eq:deligne} for the coefficients of the Hecke eigenforms we find
\begin{align}\label{eq:coeff:cusp:L16}
    \vert c_{2,m} \vert \leq (1+165+1)\cdot 2 \cdot m^{6} = 334 \cdot m^{6} \quad \text{and} \quad \vert c_{3,m} \vert \leq 2 \cdot m^{6} 
\end{align}
for the coefficients of the cusp forms $C^{(2)}_{12}(\tau)$ and $C^{(3)}_{12}(\tau)$ respectively.
Note that here we could strengthen the bound on $C^{{3}}_{12}$, as $C^{{3}}_{12}(\tau) = H^{(2)}_{12} = \Delta(3\tau)$, where the latter is a Hecke eigenform for $\Scal_{12}(\SL_2(\Z))$, so its coefficients are bound by $d(m/3) \cdot (m/3)^{11/2} \leq 2 \cdot 3^{-6} m^6$, rather than by $2 \cdot m^{6}$. 
However, we will not make use of this strengthened version fur the current application.

\subsubsection{Eigenvalues of the Hessian of $L_{16}$}

Explicitly, by Theorem \ref{thm:maintheorem}, the eigenvalues of the Hessian are
\begin{equation}
  \begin{aligned}
    \lambda_j = \frac{1}{288}\sum_{m = 1}^{\infty} \left(\left(\sum_{i = 2}^{3} \frac{1}{2}c_{i,m}\left(288\lambda_{i,j}- 8 i^2 a_i \right)\right) + 2 a_m m(2 m - 9)\right) e^{-2 m}
  \end{aligned}
\end{equation}
corresponding to the common eigenspaces $E_1$, $E_2$, and $E_3$ of $Q_4$ and $Q_6$.
Note that indeed $Q_2 = 0$, as $L_{16}$ does not have vectors of squared length $2$, and therefore does not appear in the above computation.

Before using the previously found bounds for a rigorous verification of the fact that $L_{16}$ is a local maximum for the Gaussian potential given by $\alpha = 1$, we plot the eigenvalues of the Hessian $\nabla^2 \Ecal(\alpha,L_{16})$ at $L_{16}$.
For this we use the first 200 summands of \eqref{eq:hessian:eigenvalues}, the result is depicted in Figure \ref{fig:EigVaL3D16.3}.
\begin{figure}[htb]
    \centering
    \includegraphics[width=.8\linewidth]{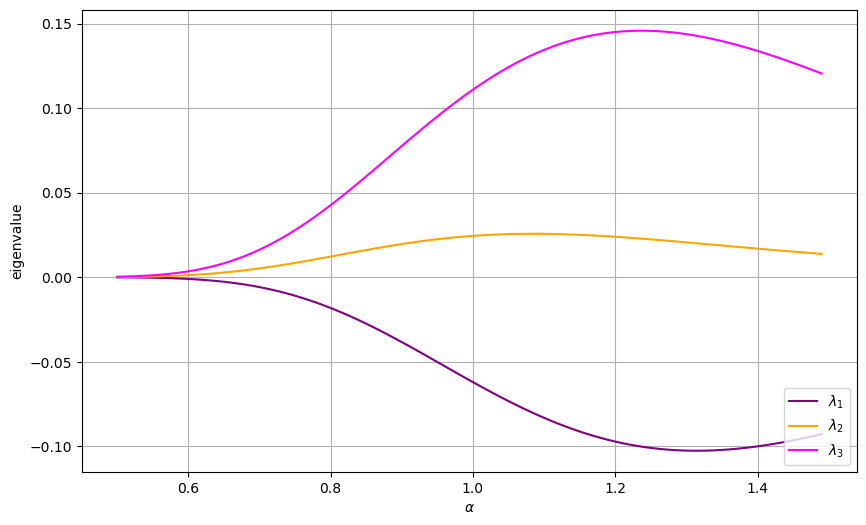}
    \caption{Eigenvalues of the Hessian for $L_{16}$}
    \label{fig:EigVaL3D16.3}
\end{figure}

We provide a rigorous estimate at $\alpha = 1$ to confirm the numerical simulation.
Specifically, we provide two estimates.
First, we show that the numerically negative eigenvalue $\lambda_1$, which is associated to the common eigenspace $E_1$ with the eigenvalues $\lambda_{2,1} = 0$ and $\lambda_{3,1} = 4320$, and printed as the violet curve, is indeed strictly negative.
Second, we show that the numerically smaller positive eigenvalue $\lambda_2$, which is associated to the common eigenspace $E_2$ with the eigenvalues $\lambda_{2,3} = 72$ and $\lambda_{3,3} = 2592$, and printed as the orange curve, is indeed strictly positive.

We first compute the sum of the first $20$ summands as an approximation for $\lambda_1$, $\lambda_2$, and $\lambda_3$:
\[
  \lambda_1 \approx -0.062,\quad \lambda_2 \approx 0.024,\quad \text{and} \quad \lambda_3 \approx 0.111.
\]
To provide a rigorous proof of negativity and positivity we now compute bounds on the tails $T_1$, $T_2$, and $T_3$, of the series defining $\lambda_1$, $\lambda_2$, and $\lambda_3$ respectively.
These bounds will show that the tail is only affecting the eigenvalue to the order of $10^{-5}$.

For $\lambda_1$, we use the explicit eigenvalues $\lambda_{2,1} = 0$ and $\lambda_{3,1} = 4320$, and obtain that
\begin{align*}
  |T_1| \leq &\sum_{m = 21}^{\infty} \left(\tfrac{1}{2}\cdot 288(0 \cdot |c_{2,m}| + 4320 \cdot |c_{3,m}|) +  (2m)^2 |a_m| \right) e^{-2 m}\\
  = &\sum_{m = 21}^{\infty} \left( 622080 \cdot \vert c_{3,m} \vert +  4 \cdot \vert a_m\vert m^2  \right)  e^{-2 m} \\
  \leq\ &1.3 \cdot 10^6 \sum_{m = 21}^{\infty} m^6 e^{-2 m} + 1.4 \cdot 10^{4} \sum_{m = 21}^{\infty} m^9e ^{-2 m}
\end{align*}
where the inequality is obtained by using the estimates on the coefficients in \eqref{eq:coeff:cusp:L16} and \eqref{eq:theta:L16}. 
Similarly we obtain 
\begin{align*}
  |T_2| \leq &\sum_{m = 21}^{\infty} \left(\tfrac{1}{2}\cdot 288(72 \cdot |c_{2,m}| + 3456 \cdot |c_{3,m}|) +  (2m)^2 |a_m| \right) e^{-2 m}\\
  \leq\ &4.5 \cdot 10^6 \sum_{m = 21}^{\infty} m^6 e^{-2 m} + 1.4 \cdot 10^{4} \sum_{m = 21}^{\infty} m^9e ^{-2 m}
\end{align*}
for $\lambda_2$ with $\lambda_{2,2} = 72$ and $\lambda_{3,2} = 3456$, and 
\begin{align*}
  |T_3| \leq &\sum_{m = 21}^{\infty} \left(\tfrac{1}{2}\cdot 288(144 \cdot |c_{2,m}| + 2592 \cdot |c_{3,m}|) +  (2m)^2 |a_m| \right) e^{-2 m}\\
  \leq\ &7.7 \cdot 10^6 \sum_{m = 21}^{\infty} m^6 e^{-2 m} + 1.4 \cdot 10^{4} \sum_{m = 21}^{\infty} m^9e ^{-2 m}
\end{align*}
for $\lambda_2$ with $\lambda_{2,3} = 144$ and $\lambda_{3,3} = 2592$.\smallskip

To conclude the estimation on the tails, we use the following lemma (Lemma 2.2 in \cite{Heimendahl2023}), which provides an analytic expression for the remaining infinite series.
\begin{lemma}\label{lem:approx:tail}
    For $j \geq \frac{k}{2\alpha}$, we have
    \begin{align*}
        \sum_{m=j}^{\infty} m^k e^{-2\alpha m} \leq j^k e^{-2\alpha j} + (2\alpha)^{-(k+1)} \Gamma(k+1, 2\alpha j),
    \end{align*}
    where
    \begin{align*}
        \Gamma(s, x) = \int_{x}^{\infty} t^{s-1} e^{-t} \, dt
    \end{align*}
    is the incomplete gamma function.
\end{lemma}

Applying Lemma \ref{lem:approx:tail} to the series in $T_1$, $T_2$, and $T_3$ yields
\begin{align*}
    &\sum_{m = 21}^{\infty} m^6 e^{-2 m} \leq 7.8 \cdot 10^{-11}, \\
    &\sum_{m = 21}^{\infty} m^9e ^{-2 m} \leq 7.5 \cdot 10^{-7}.
\end{align*}
With this we obtain the final bounds
\begin{align*}
    |T_1| \leq &\frac{1}{288}(1.3 \cdot 10^{6} \cdot 7.8 \cdot 10^{-11} + 1.4 \cdot 10^4 \cdot 7.5 \cdot 10^{-7}) \\
    \leq\ &3.7 \cdot 10^{-5} < 0.062,
\end{align*}
and analogously
\begin{align*}
    |T_2| \leq 3.8 \cdot 10^{-5} < 0.024 \quad \text{and} \quad |T_3| \leq 3.9 \cdot 10^{-5} < 0.111.
\end{align*}
This concludes the proof.
\end{proof}

\subsection{Extremal modular lattices which are not universally critical} \label{sec:noncritical}

While any extremal modular lattice on which all shells are spherical designs of strength at least $2$ is critical for the energy function $\Ecal(\alpha,L)$ for arbitrary $\alpha$ this is no longer true once the shells do not have the desired design strength. 
One explicit example among even unimodular lattices was already discussed in \cite{Heimendahl2023}, where the authors investigated a particular non-extremal even unimodular lattice in dimension $32$.

Here we will  provide an explicit example of an extremal $2$-modular lattice in dimension $12$ which is not critical for energy.
Note that the strategy we describe, utilizing the fact that the shells do not form spherical $2$-designs, can be applied to other lattices.

For the current purpose we choose an extremal $2$-modular lattice $L_{12}$ in dimension $12$ with theta series
\begin{equation} \label{eq:Lnoncritical:theta}
  \Theta_{L_{12}}(\tau) = 1 + 72q + 1800q^2 + 17468q^3 + 0(q^4).
\end{equation}
An explicit realization of such a lattice can be found in \cite{Nebe}, more precisely as the third entry in the file ``\texttt{2\_dim12.dat}''.

Note that, up to isometry, there are two more such lattices.
For the first such lattice, the first entry in ``\texttt{2\_dim12.dat}'' in \cite{Nebe} all shells form spherical $3$-designs; to test this one can use the design test based on harmonic Molien series as in Goethals and Seidel \cite{Goethals1981}, which we explained in Section \ref{sec:spherical:shells}.
So this lattice \emph{is} a universal critical point.

For the remaining lattice, the second entry in the above-mentioned file, we can (numerically) observe that it is not critical, the same method as illustrated below can in principle be used to give a rigorous proof, if desired.

\begin{proposition} \label{prop:noncrit}
  For $\alpha = 10$ the lattice $L_{12}$ is not a critical point for $\Ecal(10,L)$.
\end{proposition}

\begin{proof}
We recall that the gradient of energy at $L_{12}$ is given by 
\begin{align}\label{eq:gradient:2}
    \langle \nabla \Ecal(\alpha,L_{12}), H\rangle = - \alpha \sum_{m > 0} e^{-\alpha m} \sum_{x \in L(2m)} H[x].
\end{align}
We construct a symmetric $n\times n$ matrix $H$ with trace zero, such that 
\begin{align*}
    \langle \nabla \Ecal(\alpha,L_{12}), H\rangle \neq 0
\end{align*}
for some $\alpha$ (we will explicitly choose $\alpha = 10$ for a rigorous estimate).
For this we split \eqref{eq:gradient:2} into contributions of the shell $L_{12}(2)$ and the remainder:
\begin{align}\label{noDesignGradient}
    \langle \nabla \Ecal(\alpha,L_{12}), H\rangle = - \alpha e^{-2\alpha} \left( \sum_{x \in L_{12}(2)} H[x]\right) - \alpha \left( \sum_{m \geq 2} \sum_{x \in L_{12}(2m)} H[x] e^{-\alpha m}\right).
\end{align}
Specifically at $L_{12}$ we evaluate
\begin{align}\label{noDesignL2Sum}
    \sum_{x \in L_{12}(2)} H[x] = \left\langle H, \sum_{x \in L_{12}(2)} x x^{\sf{T}} \right\rangle
\end{align}
with
\begin{align*}
    \sum_{x \in L(2)} x x^{\sf{T}}  = \diag(20,4,20,20,20,20,4,20,4,4,4,4).
\end{align*}
This shows in particular that the shells of $L_{12}$ do not form spherical $2$-designs (c.f \eqref{eq:spherical-2-design}).
From this information we make the Ansatz to choose the symmetric $12 \times 12$ matrix 
\begin{align} \label{noDesignH}
    H = \diag(-1,1,-1,-1,-1,-1,1,-1,1,1,1,1)
\end{align}
with trace zero. 
Then
\begin{align*}
    \sum_{x \in L_{12}(2)} H[x] = - 6 \cdot 20 + 6 \cdot 4 = - 96 \neq 0.
\end{align*}
We now use the Rayleigh-Ritz principle, which provides a relation between $H[x]$, $\|x\|^2$, and the eigenvalues of $H$ by
\begin{align*}
    -1 = \lambda_{\min}(H) \leq \frac{H[x]}{\|x\|^2}\leq \lambda_{\max}(H) = 1.
\end{align*}
With this we bound the tail in \eqref{noDesignGradient}. 

Let $a_m = \vert L_{12}(2m)\vert$ be the $m$-th coefficient of the theta series $\Theta_{L_{12}}$. 
Then 
\begin{align*}
    -\sum_{m \geq 2} a_m \cdot 2m \cdot e^{-2\alpha m} \leq \sum_{m \geq 2} \sum_{x \in L_{12}(2m)} H[x] e^{-2\alpha m} \leq \sum_{m \geq 2} a_m \cdot 2m \cdot e^{-2\alpha m}.
\end{align*}
Now it suffices to show 
\begin{align}\label{noDesignSuffice}
    \sum_{m \geq 2} a_m \cdot 2m \cdot e^{-2\alpha m} < 96\cdot e^{-2\alpha}.
\end{align}
This follows the same lines as the computation in Section \ref{sec:2design}.
Note that $\Theta_{L_{12}}$ is a modular form of weight 6 for $\Gamma_0(2)$. 
The space $\Mcal_6(\Gamma_0(2))$ contains no cusp forms and is therefore equal to its two-dimensional subspace of Eisenstein series $\Eis_6(\Gamma_0(2))$. 

We use the basis $E^{(1)}_6, E^{(2)}_6$ of $\Eis_6(\Gamma_0(2))$ given in \eqref{eq:eisenstein:basis}.
Then
\begin{align*}
  \Theta_{L_12}(\tau) = E^{(1)}_6(\tau) + 72 E^{(2)}_6(\tau).
\end{align*}
To find an upper bound for the coefficients of $\Theta_{L_{12}}$, we use the estimate in \eqref{eq:coeff:bound:Eisenstein} for the coefficients of $E^{(1)}_6$ and $E^{(2)}_6$.
From this we get 
\begin{align*}
  |b^{(1)}_m|,|b^{(2)}_m| \leq 523 m^5,
\end{align*}
and so the coefficients $a_m$ are bounded by
\begin{align*}
    \vert a_m \vert \leq (1+72)\cdot 523 \cdot m^5 = 38179 \cdot m^5.
\end{align*}
Now we explicitly choose $\alpha = 10$ to give a rigorous argument. 
Evaluating the left-hand side of (\ref{noDesignSuffice}) for $\alpha = 10$ results in
\begin{align*}
    \sum_{m \geq 2} a_m \cdot 2m \cdot e^{-20 m} \leq 76358 \sum_{m \geq 2} m^6 \cdot e^{-20m}.
\end{align*}
Applying Lemma \ref{lem:approx:tail} yields 
\begin{align*}
    \sum_{m \geq 2} m^6 \cdot e^{-20m} \leq 3 \cdot 10^{-16},
\end{align*}
and together we get 
\begin{equation*}
    \sum_{m \geq 2} a_m \cdot 2m \cdot e^{-20 m} \leq 76358 \cdot 3 \cdot 10^{-16} < 96 \cdot e^{-20}.
\end{equation*}
We have thus shown that $\langle \nabla \Ecal(10,L_{12}), H\rangle \neq 0$ for $H$ as in (\ref{noDesignH}), proving that the gradient $\nabla \Ecal(10,L_{12})$ of the energy function is not zero for the lattice $L_{12}$, and it is therefore not a critical point for $\alpha = 10$.
\end{proof}

\subsection{Numerical considerations} \label{sec:numerical}
While we only discussed a number of examples in detail, we performed a numerical investigation of all extremal $2$-modular lattices up to dimension $20$ and $3$-modular lattices up to dimension $18$.
We briefly collect the qualitative behavior of these lattices (in the space of lattices of fixed density) in the Gaussian core model, the techniques used in the preceding examples can be used to rigorously verify these observations, if necessary.

The relevant data is presented in Tables \ref{tab:numerics:modular:lattices:2} and \ref{tab:numerics:modular:lattices:3} for $2$ and $3$-modular lattices separately.
We collect for each dimension how such lattices behave in the Gaussian core model (in the space of lattices of fixed density) according to our numerical computations, and give a reference to a (partial) proof of the claims, if it is available.
If there are multiple isometry classes of such lattices in a given dimension it is usually the case that the number of distinct eigenvalues of the Hessian depends on the lattice. 

In some cases, $\ell$-modular lattices of the same dimension behave differently:
We observed this in dimension $12$ for $2$-modular lattices, and in dimensions $8$ and $10$ for $3$-modular lattices. 
Here, the minimal design strength of the shells, as recorded in Tables \ref{tab:modular:lattices:2} and \ref{tab:modular:lattices:3}, is not big enough to guarantee that we have at least $2$-designs on every shell.
We therefore investigated the design strength in a case by case study, by computing the harmonic Molien series in every case, and using Theorem 3.12 in \cite{Goethals1981}, as explained in Section \ref{sec:spherical:shells}.
It turns out, that in all three cases, precisely one lattice is universally critical, the corresponding shells are always at least $3$-designs, while the rest are not universally critical, the design strength of some shells is $1$.
In the table we use the notation $d.i$ to refer to the $i$th lattice in dimension $d$ as given in the database \cite{Nebe}.

\begin{table}\renewcommand{\arraystretch}{1.1}
  \begin{tabular}{c c c c}\hline
    dimension & lattices & behavior & reference \\ \hline \rule{0pt}{4mm}
    $4$ & all & loc. univ. opt. & \cite{Sarnak2006a}, Theorem \ref{thm:ellmodular:loc}\\
    $8$ & all & univ. saddle point & --- \\
    $12$ & 12.1 & univ. saddle point & --- \\
     & 12.2, 12.3 & not univ. critical & Proposition \ref{prop:noncrit} \\
    $16$ & all & loc. univ. opt. & Proposition \ref{prop:BW}\\
    $20$ & all & loc. univ. opt. & Theorem \ref{thm:ellmodular:loc}\\ \hline
  \end{tabular}
  \medskip
  \caption{Numerical data on the behavior of $2$-modular extremal lattices in small dimensions.}
  \label{tab:numerics:modular:lattices:2}
\end{table}

\begin{table}\renewcommand{\arraystretch}{1.1}
  \begin{tabular}{c c c c}\hline
    dimension & lattices & behavior & reference \\ \hline \rule{0pt}{4mm}%
    $2$ & all & (loc.) univ. opt. & \cite{Montgomery1988}, Theorem \ref{thm:ellmodular:loc} \\
    $4$ & all & univ. saddle point & --- \\
    $6$ & all & univ. saddle point & --- \\
    $8$ & 8.1 & univ. saddle point & --- \\
     & 8.2 & not univ. critical & ---  \\
    $10$ & 10.1 & univ. saddle point & --- \\
     & 10.2, 10.3 & not univ. critical & ---  \\    
    $12$ & all & loc. univ. opt. & Proposition \ref{prop:CT} \\
    $14$ & all & loc. univ. opt. & Theorem \ref{thm:ellmodular:loc} \\
    $16$ & all & univ. saddle point & --- \\
    $18$ & all & univ. saddle point & --- \\\hline
  \end{tabular}
  \medskip
  \caption{Numerical data on the behavior of $3$-modular extremal lattices in small dimensions.}
  \label{tab:numerics:modular:lattices:3}
\end{table}

\subsection{A brief remark on simultaneous diagonalizability}\label{sec:simultaneous:diagonalizable}

A question of interest to us is the simultaneous diagonalizability of the quadratic forms $Q_1,\ldots,Q_{2e}$ in Theorem \ref{thm:maintheorem}. 
In all explicitly tested examples these forms satisfied this assumption, and indeed from this it follows by \eqref{eq:formelsummehx2} that then all \emph{shell forms} $Q_{2m}$ for $m \geq 0$ (c.f. \eqref{eq:crucial-quadratic-form}) are simultaneously diagonalizable.

Such $Q$ relates to the bilinear form $ b_Q $, defined in \eqref{eq:bil-form-b_Q}, which is invariant under the action of $\ort(L)$, that is $ b_Q(U G U^{\sf T},U H U^{\sf T}) = b_Q(G,H) $ for all $ S \in \ort(L) $.
The forms $Q_{1},\ldots,Q_{2e}$ are simultaneously diagonalizable if and only if the Gram matrices of these associated bilinear forms commute.

The isometry group $\ort(L)$ of the lattice $L$ acts on the space of symmetric matrices $\mathcal{S}^n$ by conjugation
\begin{align*}
	\ort(L) \times \mathcal{S}^n &\to \mathcal{S}^n   \\
	(U,H)  &\mapsto  U H U^{\sf T}.
\end{align*}
This turns $ (\mathcal{S}^n, \langle \cdot, \cdot \rangle )$ into a unitary representation of $ \ort(L) $, meaning that the action of $ \ort(L) $ preserves the inner product $ \langle \cdot , \cdot \rangle $.
By virtue of the Riesz representation theorem, there is a linear map $T : \mathcal{S}^n \to \mathcal{S}^n$ such that
\[
	b_Q(G,H) = \langle G, T(H) \rangle
\]
and the eigenvalues of the Gram matrix of $b_Q$ coincide with the eigenvalues of $T$.
Since $ b_Q $ is invariant under the action of $ \ort(L) $, the map $ T $ commutes with the action of $ \ort(L)$, i.e.
\begin{equation*}
\label{eq:commutativity-property-of-E}
T(U H U^{\sf T}) = U T(H)U^{\sf T} \quad   \textup{ for all }    U \in \ort(L),
\end{equation*}
hence, $ T $ is an $\ort(L)$-equivariant endomorphism of $\mathcal{S}^n$. In other words, $T$ is intertwining the representation
$ (\mathcal{S}^n, \langle \cdot, \cdot \rangle ) $ of the group $ \ort(L) $ with itself.
In this way we have a $1$-$1$-correspondence between $\ort(L)$-invariant bilinear forms $Q$ on $\mathcal{S}^n$ and $\ort(L)$-equivariant endomorphisms of $\mathcal{S}^n$.

Now it can be checked directly (say with the use of a computer algebra system) that in general this representation of $\ort(L)$ on $\Scal^n$ contains an irreducible component with multiplicity larger than $1$. 
By this it follows that there are $\ort(L)$-invariant bilinear forms which are not simultaneously diagonalizable.
So the shell forms $Q_{2m}$ we are interested in seem to be rather special. 
In order to understand them better it would be nice to identify a corresponding commutative subalgebra of $\ort(L)$-equivariant endomorphisms of the $\ort(L)$-representation $\Scal^n$.

\section{Outlook} \label{sec:outlook}
The strategy described in Section \ref{sec:Strategy} and carried out for some examples in Section \ref{sec:results} can likely be applied to other extremal (strongly) $\ell$-modular lattices as discussed in \cite{Quebbemann1995}, \cite{Quebbemann1997}, and \cite{Scharlau1999}.
However, some adjustments have to be made:
Firstly, one needs explicit information about the spaces $\Scal_{n/2+4}(\Gamma_0(\ell,\chi))$, in particular if $\chi$ is non-trivial, to write down the data needed for an application of Theorem \ref{thm:maintheorem} (or a variant of it).
This includes the dimension of this space and a basis as in \eqref{eq:row:echelon}, or it would involve more complicated computations following \eqref{eq:thetaLph4}.
Next, one needs to carry out computations to bound the coefficients of the cusp forms in $\Scal_{k}(\Gamma_0(\ell),\chi)$, for which one can use the method described in Section \ref{sec:cuspform:coeff}, if one can find a decomposition into Hecke eigenforms.
Analogous bounds for larger $\ell$ and characters $\chi$ would allow for rigorous computations as presented throughout Section \ref{sec:results}.

\section*{Acknowledgements}
We thank Matthew de Courcy-Ireland and the anonymous referee for pointing out how to bound the coefficients of cusp forms using the bound on Hecke eigenforms coming from Deligne's theorem, improving the resulting bounds by certain orders of magnitude compared to the previously used global bounds.

We further thank the anonymous referee for the helpful comments, suggestions, and corrections on the manuscript.

We thank Henry Cohn for pointing out an error in an earlier version.


\begin{thebibliography}{16}

\bibitem{Bachoc2001}
  C. Bachoc, B. Venkov, Modular forms, lattices and spherical designs,
  pages 87--111 in: R\'eseaux euclidiens, designs sph\'eriques et formes modulaires,
Monogr. Enseign. Math., 37, 2001.

\bibitem{Barnes1959}
  E.S. Barnes, G.E. Wall, Some extreme forms defined in terms of Abelian groups,
  J. Aust. Math. Soc. 1 (1959), 47--63.

\bibitem{Choi2018}
  S. Choi, B. Im, 
  Bounds for the coefficients of cusp forms for $\Gamma_0(3)$,
  J. Number Theory 188 (2018), 48--70.

\bibitem{Cohn2007}
  H. Cohn, A. Kumar, Universally optimal distribution of points on spheres,
  J. Amer. Math. Soc. 20 (2007), 99--148.

\bibitem{Cohn2009}
  H. Cohn, A. Kumar, A. Sch\"urmann, Ground states and formal duality relations in the Gaussian core model,
  Phys. Rev. E 80, 061116 (2009).

\bibitem{Cohn2019a}
  H. Cohn, A. Kumar, S.D. Miller, D. Radchenko, M. Viazovska,
  Universal optimality of the $E_8$ and Leech lattices and interpolation
  formulas, Ann. Math. 196 (2022), 983--1082.

\bibitem{Conway1988a}
J.H. Conway, N.J.A. Sloane, Sphere packings, lattices, and groups,
Springer, 1988.

\bibitem{Coulangeon2006a}
  R. Coulangeon,
  Spherical designs and zeta functions of lattices, Int. Math. Res. Not. IMRN (2006), 1--16.

\bibitem{Coulangeon2012a}
  R. Coulangeon, A. Sch\"urmann, Energy minimization, periodic sets
  and spherical designs, Int. Math. Res. Not. IMRN (2012), 829--848.

\bibitem{Coxeter1953}
  H.S.M. Coxeter, J.A. Todd, An extreme duodenary form,
  Can. J. Math. 5 (1953), 384--392.

\bibitem{Ebeling1994a}
  W. Ebeling, Lattices and Codes, Vieweg, 1994.

\bibitem{deCourcyIreland:sds:2024}
  M. de Courcy-Ireland, M. Dostert, M. Viazovska, Six-dimensional sphere packing and linear programming,
  Mathematics of Computation \textbf{93} 348 (2024), 1993--2029.

\bibitem{Deligne1974}
  P. Deligne, La conjecture de Weil: I,
  Inst. Hautes \'{E}tudes Sci. Publ. Math. \textbf{43} (1973), 273--307.

\bibitem{Diamond2005}
  F. Diamond, J.M. Shurman, A first course in modular forms,
  Springer, 2005.

\bibitem{Goethals1981}
  J.M. Goethals, J.J. Seidel, 
  Cubature Formulae, Polytopes, and Spherical Designs,
  The Geometric Vein, (1981), Springer New York, 203--218.

\bibitem{Heimendahl2023}
  A. Heimendahl, A. Marafioti, A. Thiemeyer, F. Vallentin, M.C. Zimmermann, 
  Critical even unimodular lattices in the Gaussian core model,
  Int. Math. Res. Not. IMRN (2023), 5352--5396.

\bibitem{Joharian2024}
  A. Joharian, Critical modular lattices of 2 and 3 in the Gaussian core model,
  Master's thesis, Universit\"at zu K\"oln, 2024.

\bibitem{Jenkins2011a}
  P. Jenkins, J. Rouse, Bounds for coefficients of cusp forms and
  extremal lattices, Bull. London Math. Soc., 43 (2011), 927--938.

\bibitem{Jenkins2015}
  P. Jenkins, K. Pratt, Coefficient bounds for level 2 cusp forms,
  Acta Arith. 168 (2015), 341--367.

\bibitem{LMFDB}
  The LMFDB Coollaboration, The L-function and modular forms database,
  \url{https://www.lmfdb.org}, accessed 30 July 2025.  

\bibitem{MAGMA}
W. Bosma, J.J. Cannon, C. Fieker, A. Steel (eds.),
Handbook of Magma functions, Edition 2.16 (2010), 5017 pages.

\bibitem{Miller1975}
  V.S. Miller, Diophantine and $p$-adic analysis of elliptic curves and modular forms, 
  Ph.D. thesis, Harvard University, 1975.

\bibitem{Montgomery1988}
  H.L. Montgomery, Minimal theta functions,
  Glasgow Math. J. 30 (1988), 75--85.

\bibitem{Nebe2013a}
G. Nebe,
Boris Venkov's theory of lattices and spherical designs, pages 1--19
in: Diophantine methods, lattices, and arithmetic theory of quadratic forms,
Contemp. Math., 587, Amer. Math. Soc., 2013.

\bibitem{Nebe}
G. Nebe, N.J.A. Sloane, A Catalogue of Lattices,
\url{https://www.math.rwth-aachen.de/~Gabriele.Nebe/LATTICES/ExtLat.html}

\bibitem{PARI}
  The PARI~Group, PARI/GP version \texttt{2.17.2}, Univ. Bordeaux, 2025,
  \url{http://pari.math.u-bordeaux.fr/}.

\bibitem{Quebbemann1995}
  H.G. Quebbemann, Modular lattices in Euclidean spaces,
  J. Number Theory 54 (1995), 190--202.

\bibitem{Quebbemann1997}
  H.G. Quebbemann, Atkin-Lehner eigenforms and strongly modular lattices,
  L'Enseignement Math\'ematique 43 (1997), 55--65.

\bibitem{Regev2016a}
  O. Regev, N. Stephens-Davidowitz,
  A reverse Minkowski theorem,
  Ann. Math. 199 (2024), 1--49.


\bibitem{Sarnak2006a}
  P. Sarnak, A. Str\"ombergsson,
  Minima of Epstein's zeta function and heights of flat tori,
  Invent. Math. 165 (2006), 115--151.

\bibitem{Scharlau1999}
  R. Scharlau, R. Schulze-Pillot, Extremal lattices,
  Algorithmic Algebra and Number Theory: Selected papers from a conference held at the University of Heidelberg in October 1997 (1999), 139--170.

\bibitem{Serre1973a}
  J.-P. Serre, A Course in Arithmetic, Springer-Verlag, 1973.

\bibitem{Stillinger1976a} F.H. Stillinger, Phase transitions in the
  Gaussian core system, J. Chem. Phys. 65 (1976), 3968--3974.

\bibitem{Torquato2008}
S. Torquato, F.H. Stillinger, New duality relations for classical ground states,
Phys. Rev. Lett. 100, 020602 (2008).

\bibitem{Venkov2001a} B.B. Venkov, R\'eseaux et designs sph\'eriques,
  pages 10--86 in: R\'eseaux euclidiens, designs sph\'eriques et formes modulaires,
Monogr. Enseign. Math., 37, 2001.

\bibitem{Voronoi1907}
  G.F. Voronoi, Nouvelles applications des param\`etres continus \`a la th\'eorie des formes quadratiques. Premier M\'emoire. Sur quelques propri\'et\'es des formes quadratiques positives
  parfaites,
  J. Reine Angew. Math. 133 (1907), 97--178.

\bibitem{Wang2023}
  W. Wang, On the Miller basis for the space of cusp forms,
  Int. J. Number Theory. 19 (2023), 1075--1095.

\bibitem{Widder1941a} D.V. Widder, The Laplace Transform, Princeton
University Press, 1941.

\end{thebibliography}
\end{document}